\documentclass[12pt]{amsart}
\usepackage{cases}
\usepackage{amsmath}
\usepackage{mathrsfs}
\usepackage{amssymb}
\usepackage{amsfonts}
\usepackage{graphicx}
\usepackage{stmaryrd}
\usepackage{abstract}
\usepackage{pdfsync}
\usepackage{hyperref}
\usepackage{pgfplots}
\usepackage{tabularx}
\usepackage{geometry}
\newgeometry{left=2cm,right=2cm,top=2cm,bottom=1.8cm}

\newtheorem{theorem}{Theorem}[section]

\newtheorem{lemma}[theorem]{Lemma}

\theoremstyle{definition}
\newtheorem{definition}[theorem]{Definition}
\newtheorem{asmp}[theorem]{Assumption}

\theoremstyle{remark}
\newtheorem{remark}[theorem]{Remark}
\newtheorem*{remark*}{Remark}

\numberwithin{equation}{section}
\numberwithin{case}{section}
\numberwithin{subcase}{case}
\numberwithin{subsubcase}{subcase}

\newcommand{\bm}[1]{{\mbox{\boldmath$#1$}}}
\def\MR#1{\href{http://www.ams.org/mathscinet-getitem?mr=#1}{MR#1}}

\begin{document}
\title[Smith Normal Form Distribution of a Random Integer Matrix]{The Smith Normal Form Distribution of a Random Integer Matrix}

\date{June 11, 2015}

\author{Yinghui Wang} \email{yinghui@math.columbia.edu, yinghui@alum.mit.edu}
\address{Department of Mathematics, Columbia University, New York, New York 10027}

\author{Richard P. Stanley}
\email{rstan@math.mit.edu}
\address{Department of Mathematics, Massachusetts Institute of Technology, Cambridge, Massachusetts 02139}

\thanks{{\it Acknowledgements:} The authors are grateful to Professor Bjorn Poonen for advice on the literature on the subject of this paper. The second author was partially supported by NSF grant DMS-1068625.}

\maketitle

\begin{abstract}
We show that the density $\mu$ of the Smith normal form (SNF) of a
random integer matrix exists and equals a product of densities
$\mu_{p^s}$ of SNF over $\mathbb{Z}/p^s\mathbb{Z}$ with $p$ a prime
and $s$ some positive integer.   
Our approach is to connect the SNF of a matrix with the greatest
common divisors (gcds) of certain polynomials of matrix entries, and
develop the theory of multi-gcd distribution of polynomial values at a
random integer vector. 
We also derive a formula for $\mu_{p^s}$ and compute the density $\mu$
for several interesting types of sets.  
Finally, we determine the maximum and minimum of $\mu_{p^s}$ and
establish its monotonicity properties and limiting behaviors. 

\end{abstract}

\section{Introduction}

Let $M$ be a nonzero $n\times m$ matrix over a commutative ring  $R$
(with identity), and $r$ be the rank of $M$.  
If there exist invertible $n \times n$ and $m \times m$ matrices $P$ and $Q$ such that the product $PMQ$ is a diagonal matrix with diagonal entries $d_1,d_2,\dots, d_r,0,0,\dots,0$ satisfying that
$d_i \mid d_{i+1} $ for all $1 \le i \le  r-1$, 
then $PMQ$ is the {\it Smith normal form (SNF)} of $M$.
In general, the SNF does not exist. It does exist when $R$ is a
\emph{principal ideal ring}, i.e., a ring (not necessarily an integral
domain) for which every ideal is principal.  
This class of rings includes the integers $\mathbb{Z}$ and their quotients $\mathbb{Z}/q\mathbb{Z}$, which are the rings of interest to us here.  
In fact, for the rings $\mathbb{Z}/q\mathbb{Z}$ we will be particularly concerned with the case $q=p^s$, a prime power. 
For principal ideal rings, the diagonal entries are uniquely determined (up to multiplication by a unit) by $g_{i-1}d_i = g_i$ $(1 \le i \le  r)$, 
where $g_0=1$ and $g_i$ is the greatest common divisor (gcd) of all $i\times i$ minors of $M$.
We have the following correspondence between the SNF and the cokernel of $M$:
${\mathrm{coker}}\,M \simeq  R/d_1R\oplus R/d_2R\oplus \cdots \oplus R/d_rR \oplus R^{n-r}$.

There has been a huge amount of research on eigenvalues of random matrices over a field (see, e.g., \cite{ABD}, \cite{AGZ},  \cite{F}, \cite{M}). 
Less attention has been paid to the SNF of a random matrix over a principal ideal ring (or more general rings for which SNF always exists).
Some basic results in this area are known, 
but they appear in papers not focused on SNF per se. 
We develop the theory in a systematic way, collecting previous work in this area, sometimes with simplified proofs, and providing some new results. 

\smallskip
We shall define the {\it density} $\mu$ of SNF of a random $n\times m$ integer matrix as the limit (if exists) as $k\to \infty$ of $\mu^{(k)}$, the density of SNF of a random $n\times m$ matrix with entries independent and uniformly distributed over $\{-k,-k+1,\dots,k\}$  
(see Definition \ref{def:mu} below for a precise definition).  

As a motivating example, 
the probability that $d_1=1$ for a random $n\times m$ integer matrix is the probability that the $nm$ matrix entries are relatively prime, or equivalently, that $nm$ random integers are relatively prime, and thus equals $1/\zeta(nm)$, where $\zeta(\cdot)$ is the Riemann zeta function. 

\smallskip
If we regard the minors of an $n\times m$ matrix as polynomials of the $nm$ matrix entries with integer coefficients, then the SNF of a matrix is uniquely determined by the gcds of the values of these polynomials (recall the definition of SNF from the beginning). 
This inspires us to study the theory of multi-gcd distribution of polynomial values.

Given a collection of relatively prime polynomials in $\mathbb{Z}[x_1,x_2,\dots,x_d]$, 
let $g(x)$ be the gcd of the values of these polynomials at $x=(x_1,x_2,\dots,x_d)$. 
We shall define the {\it density} $\lambda$ of $g(x)$ of a random $d$-dimensional integer vector $x$ as the limit (if exists) as $k\to \infty$ of $\lambda^{(k)}$, the density of $g(x)$ with $x$ uniformly distributed over $\{-k,-k+1,\dots,k\}^d$ (see Definition \ref{def:lambda} for a precise definition). 

In the spirit of previous work in number theory such as \cite{E}, \cite{P}, \cite{PS} and the Cohen-Lenstra heuristics (\cite{CL82}, \cite{CL83}), 
one might conjecture that $\lambda$ exists and equals the product of density $\lambda_{p}$ of $g(x)$ over $(\mathbb{Z}/p\mathbb{Z})^d$ over all primes $p$. 
In fact, we will prove this conjecture with the more general density  $\lambda_{p^s}$ of $g(x)$ over $\mathbb{Z}/p^s\mathbb{Z}$ for sets of form \eqref{eq:Z} (see Theorem \ref{thm:lambda(Z)}), with the aid of a result in number theory \cite[Lemma 21]{PS}. 
Note that the special case that $s=0$ or $1$  follows from \cite[Theorem 2.3]{E} directly. 
In particular, this result applies to the probability that $g(x)=1$\,, in other words, that the polynomial values are relatively prime.
Furthermore, all these results hold for the multi-gcd distribution of polynomial values, namely, when $g(x)$ is a vector whose components are the gcds of the values of given collections of polynomials at $x$.

\smallskip
Then we apply this theory to the SNF distribution of a random integer
matrix to show that the density $\mu$ (of SNF of a random $n\times m$
integer matrix) equals a product of some densities $\mu_{p^s}$ of SNF
over $\mathbb{Z}/p^s\mathbb{Z}$ for sets of form \eqref{eq:S} (Theorem
\ref{thm:mu(S)}).  We also derive a formula for $\mu_{p^s}$ (Theorem
\ref{thm:p^s}), which allows us to compute $\mu_{p^s}$ and hence $\mu$
explicitly (Theorem \ref{thm:mu}). Some special cases of this formula
coincide with \cite[Exercise 1.192(b)]{S} and \cite[pp.\,233,
  236]{FW}.  Another paper related to our work is \cite{W}.

On the strength of these results, we determine the value of $\mu$ for
some interesting types of sets, specifically, matrices with first few
diagonal entries given, matrices with diagonal entries all equal to
$1$, and square matrices with at most $\ell\,(=1,2,\dots,n)$ diagonal
entries not equal to $1$, i.e., whose corresponding cokernel has at
most $\ell$ generators; further, for the last set we establish the
asymptotics of $\mu$ as $\ell\to \infty$\,.  In the case of $\ell=1$
(which is equivalent to the matrix having a cyclic cokernel),
our results echo those of Ekedahl \cite[Section 3]{E} 
via a different approach.  We also show that the probability that a
random integer matrix is full rank is $1$, and that $\mu$ of a finite
set is $0$.

Additionally, we find the maximum and minimum of $\mu_{p^s}(D)$ over
all diagonal matrices $D$\,; whereas regarding it as a function of
$p,s,m,n$ and $D$, we find its monotonicity properties and limiting
behaviors. 

\smallskip
The remainder of this paper is organized as follows. 
Section \ref{sec:gcd} develops the theory of multi-gcd distribution of polynomial values.
Section \ref{sec:SNF} applies this theory to the SNF distribution and derives a formula for $\mu_{p^s}$.
Section \ref{sec:app} computes the density $\mu$ for several types of sets. Finally, Section \ref{sec:properties} determines the maximum and minimum of $\mu_{p^s}$ and discusses its monotonicity properties and limiting behaviors.

\smallskip
We shall assume that throughout this paper, $p$ represents a prime, $p_j$ is the $j$-th smallest prime, and $\prod_p$ means a product over all primes $p$.   

\section{Multi-gcd Distribution of Polynomial Values}\label{sec:gcd}
Suppose that $d$ and $h$ are positive integers and $F_1,F_2,\dots,F_h\in \mathbb{Z}[x_1,x_2,\dots,x_d]$ are nonzero polynomials. 
Let $$g(x):=\gcd(F_1(x),F_2(x),\cdots,F_h(x))\,,\quad x\in \mathbb{Z}^d$$
be the gcd of the values $F_1(x),F_2(x),\dots,F_h(x)$, and $g(x)=0$ if $F_j(x)=0$ for all $1\le j\le h$.

We shall define the {\it density of $g(x)$ of a random $d$-dimensional integer vector} $x$ as the limit (if exists) of the density of $g(x)$ with $x$ uniformly distributed over $\{-k,-k+1,\dots,k\}^d:=\mathbb{Z}_{(k)}^d$ as $k\to \infty$\,, precisely as follows.

\begin{definition}\label{def:lambda}
(i) For $\mathcal{Z}\subseteq \mathbb{Z}$\,, 
we denote by $\lambda^{(k)}(\mathcal{Z})$ the probability that $g(x)\in \mathcal{Z}$ with $x$ uniformly distributed over $\mathbb{Z}_{(k)}^d$. 
If \,$\lim_{k\to \infty} \lambda^{(k)}(\mathcal{Z})=\lambda(\mathcal{Z})$ exists, then we say that  the {\em probability that $g(x)\in \mathcal{Z}$ with $x$ a random $d$-dimensional integer vector} is $\lambda(\mathcal{Z})$. If this is the case, then $\lambda(\mathcal{Z})\in [0,1]$ since  $\lambda^{(k)}(\mathcal{Z})\in [0,1]$ for all $k$.

\smallskip
(ii) We define similarly the {\rm gcd distribution over the ring of integers mod $p^s$}: for prime $p$ and positive integer $s$, we denote by $\lambda^{(k)}_{p^s}(\mathcal{Z})$ the probability that $g(x)\in \mathcal{Z}\, (\mathrm{mod}\ p^s)$ (up to multiplication by a unit) with $x$ uniformly distributed over $\mathbb{Z}_{(k)}^d$, and by $\lambda_{p^s}(\mathcal{Z})$ the probability that $g(x)\in \mathcal{Z}\, (\mathrm{mod}\ p^s)$ (up to multiplication by a unit) with $x$ uniformly distributed over $(\mathbb{Z}/p^s \mathbb{Z})^d$. 

\smallskip
More generally, for a finite set $\mathcal{P}$ of prime and positive integer pairs $(p,s)$ (with $p$ a prime  and  $s$ a positive integer), we denote  $$P_{\mathcal{P}}:=\prod_{(p,s)\in\mathcal{P}} p^s$$ 
and by $\lambda^{(k)}_{P_{\mathcal{P}}}(\mathcal{Z})$ the probability that $g(x)\in \mathcal{Z}\, (\mathrm{mod}\ P_{\mathcal{P}})$ (up to multiplication by a unit) with $x$ uniformly distributed over $\mathbb{Z}_{(k)}^d$, and by $\lambda_{P_{\mathcal{P}}}(\mathcal{Z})$ the probability that $g(x)\in \mathcal{Z}\, (\mathrm{mod}\ P_{\mathcal{P}})$ (up to multiplication by a unit) with $x$ uniformly distributed over $(\mathbb{Z}/P_{\mathcal{P}}  \mathbb{Z})^d$.  
Note that $\lambda_{P_{\mathcal{P}}}(\mathcal{Z})$ is the number of solutions to $g(x)\in \mathcal{Z}\, (\mathrm{mod}\ P_{\mathcal{P}})$ (up to multiplication by a unit) divided by $P_{\mathcal{P}}^d$. 
The situation discussed in the previous paragraph is the special case that $\mathcal{P}$ consists of only one element $(p,s)$ and $P_\mathcal{P}=p^s$.

\smallskip
(iii) The above definitions also extend to the {\rm distribution of multi-gcds}. Suppose that\,  $\mathcal{U}=\{U_i\}_{i=1}^{w}$ is a collection of $w$ nonempty subsets $U_{i}$ of $  \{F_1,F_2,\dots,F_h\}$.
Let 
\begin{equation}\label{eq:gi}
g_i(x):=\gcd\,(F(x):F\in U_i)\,,\quad x\in \mathbb{Z}^d
\end{equation}
and
$$g(x):=(g_1,g_2,\dots,g_w)(x)\in \mathbb{Z}^w,$$
then we adopt the above definitions of functions $\lambda^{(k)}$, $\lambda$, 
$\lambda^{(k)}_{P_{\mathcal{P}}}$ and $\lambda_{P_{\mathcal{P}}}$ for  $\mathcal{Z}\subseteq \mathbb{Z}^d$ with only one slight modification: replace ``up to multiplication by a unit" with ``up to multiplication of the {\rm components of $g$} by units".
\end{definition}

For convenience, we shall always assume that the notion $g(x)\in \mathcal{Z}\, (\mathrm{mod}\ P_{\mathcal{P}})$ implies the {\it equivalence of multiplication of its components by units}  and that the random vector $x$ is {\it uniformly distributed} on its range (if known, e.g., $\mathbb{Z}_{(k)}^d$ or  $(\mathbb{Z}/P_{\mathcal{P}} \mathbb{Z})^d$).

\begin{remark} 
The density $\lambda_p(\cdot)$ defined above in Definition \ref{def:lambda}\,(ii) is consistent with the normalized {\it Haar measure} on $\mathbb{Z}_p^d$\,, as in \cite{PS}.
\end{remark} 

In this section, we establish the properties of $ \lambda_{P_{\mathcal{P}}}$ and $\lambda$,  
the existence of $\lambda$, and
a connection between $\lambda$ and the $\lambda_{p^s}$'s. Then we apply these results to determine the probability that the polynomial values are relatively prime.

\subsection{\texorpdfstring{Multi-gcd Distribution over $\mathbb{Z}/P_{\mathcal{P}} \mathbb{Z}\,$}{Multi-gcd Distribution over Z/(P{\_}P)Z}}
\label{sec:lambda_{P_P}}
$ $

We show that the density $\lambda^{(k)}_{P_{\mathcal{P}}}(\cdot)$ over $\mathbb{Z}_{(k)}^d$ (defined above in Definition \ref{def:lambda}) converges to the density $\lambda_{P_{\mathcal{P}}}(\cdot)$ over $\mathbb{Z}/P_{\mathcal{P}} \mathbb{Z}$ as $k\to \infty$\,, and that  $\lambda_{P_{\mathcal{P}}}(\cdot)$ equals $\prod_{(p,s)\in \mathcal{P}} \lambda_{p^s}(\cdot)$.

\begin{theorem}\label{thm:lambda_{P_P}}
For any $\mathcal{Z}\subseteq  \mathbb{Z}^w$, we have
\begin{equation}\label{eq:lambda({z})}
\lambda_{P_{\mathcal{P}}}(\mathcal{Z})
=\sum_{z\in \mathcal{Z} \,(\mathrm{mod}\, P_{\mathcal{P}})} \lambda_{P_{\mathcal{P}}}(\{z\}) 
\end{equation}
and
\begin{equation}\label{eq:lambda_{P_P}(Z)}
\lim_{k\to \infty} \lambda^{(k)}_{P_{\mathcal{P}}}(\mathcal{Z})
=\lambda_{P_{\mathcal{P}}}(\mathcal{Z})
=\prod_{(p,s)\in \mathcal{P}} \lambda_{p^s}(\mathcal{Z}) \,.
\end{equation}
\end{theorem}
\begin{proof}
(1) The first equality \eqref{eq:lambda({z})} follows directly from Definition \ref{def:lambda}.

\smallskip
(2) For the second equality of \eqref{eq:lambda_{P_P}(Z)}, we let $N_{P_{\mathcal{P}}}(\mathcal{Z})$ be the number of $x\in (\mathbb{Z}/P_{\mathcal{P}}\mathbb{Z})^d$ for which $g(x)\in \mathcal{Z}$ (mod  $P_{\mathcal{P}}$).
The Chinese remainder theorem along with Definition \ref{def:lambda} then gives
\begin{equation*}
P_{\mathcal{P}} ^{d}\, \lambda_{P_{\mathcal{P}}}(\mathcal{Z})
=N_{P_{\mathcal{P}}}(\mathcal{Z})=
\prod_{(p,s)\in \mathcal{P}} N_{p^{s}}(\mathcal{Z})
=\prod_{(p,s)\in \mathcal{P}} p^{sd}\lambda_{p^{s}}(\mathcal{Z})
=P_{\mathcal{P}} ^{d} \prod_{(p,s)\in \mathcal{P}}\lambda_{p^{s}}(\mathcal{Z})\,.
\end{equation*}
Dividing both sides by $P_{\mathcal{P}} ^{d}$  leads to the desired equality.

\smallskip
(3) For the first equality of \eqref{eq:lambda_{P_P}(Z)}, we first observe that if $p\,|\,2k+1$, then $\lambda^{(k)}_{P_{\mathcal{P}}}(\mathcal{Z})
=\lambda_{P_{\mathcal{P}}}(\mathcal{Z})$ by definition. If $p\nmid 2k+1$, then we proceed by approximating  $2k+1$ by a multiple of $P_{\mathcal{P}}$ and estimating $\lambda^{(k)}_{P_{\mathcal{P}}}(\mathcal{Z})$ using $\lambda_{P_{\mathcal{P}}}(\mathcal{Z})$. 

\smallskip
Let $k\in \mathbb{Z}$ such that $K:=2k+1\ge P_{\mathcal{P}}$,
then there exists $q\in \mathbb{Z}_+$ such that
\begin{equation}\label{eq:q}
q\cdot P_{\mathcal{P}} \le  K < (q+1)\cdot P_{\mathcal{P}}\,. 
\end{equation}  
It follows that  for any integer $y$, there are either $q$ or $q+1$ numbers among $\mathbb{Z}_{(k)}$ that equal $y$ mod $P_{\mathcal{P}}$. 
Thus the number of $x\in \mathbb{Z}_{(k)}^d$ for which  for which $g(x)\in \mathcal{Z}$ (mod  $P_{\mathcal{P}}$) is between 
$q^{d} N'$ and $(q+1)^{d} N'$, where 
$N':=N_{P_{\mathcal{P}}}(\mathcal{Z})$, therefore
\begin{equation*}
\lambda^{(k)}_{P_{\mathcal{P}}}(\mathcal{Z})\in \left[\frac{q^{d} N'}{ K^{d} }\,, \frac{(q+1)^{d} N'}{ K^{d} }\right]:=J_k\,.
\end{equation*}
Thanks to \eqref{eq:q}, we have
\begin{equation*}
J_k \subseteq \left[\frac{q^{d} N'}{ \left[ (q+1)P_{\mathcal{P}}\right] ^{d} }\,, \frac{(q+1)^{d} N'}{ \left( qP_{\mathcal{P}}\right) ^{d} }\right] 
=\left[ \left(\frac{q}{q+1}\right) ^{d} \frac{N'}{P_{\mathcal{P}}^{d} }\,, \left(\frac{q+1}{q}\right) ^{d} \frac{N'}{P_{\mathcal{P}}^{d} } \right] ,
\end{equation*}
whose left and right endpoints both converge to $N'/P_{\mathcal{P}}^{d}$ as $q\to \infty$\,. Hence  
\begin{equation*}
\lambda^{(k)}_{P_{\mathcal{P}}}(\mathcal{Z})\to N'/P_{\mathcal{P}}^{d}
=\lambda_{P_{\mathcal{P}}}(\mathcal{Z})\,,
\end{equation*} 
as $q\to \infty$\,, or equivalently, as $k\to \infty$\,, as desired.
\end{proof}

\subsection{\texorpdfstring{Multi-gcd Distribution over $\mathbb{Z}$\,}{Multi-gcd Distribution over Z}}
\label{sec:properties of lambda}
$ $

We show some properties of the density $\lambda$ of set unions, subtractions and complements.
They will be very useful in  determining the value of $\lambda$ for specific sets (such as in Remark \ref{rmk:0}\,(iii)).

\begin{theorem}\label{thm:lambda(union)}
Suppose that $\{\mathcal{Z}_{\alpha}\}_{\alpha\in \mathcal{A}}$ are pairwise disjoint subsets of\, $\mathbb{Z}^w$ such that $\lambda(\mathcal{Z}_\alpha)$ exists for all $\alpha\in \mathcal{A}$\,. If $\mathcal{A}$ is a  finite set, then
\begin{equation*}
\lambda\left( \cup_{\alpha\in \mathcal{A}}\, \mathcal{Z}_{\alpha}\right) =\sum_{\alpha\in \mathcal{A}} \lambda(\mathcal{Z}_{\alpha})\,.
\end{equation*}
\end{theorem}

\begin{proof}
By Definition \ref{def:lambda}, we have
\begin{equation*}
\sum_{\alpha\in \mathcal{A}} \lambda(\mathcal{Z}_{\alpha})
=\sum_{\alpha\in \mathcal{A}}\lim_{k\to \infty} \lambda^{(k)}(\mathcal{Z}_{\alpha})
=\lim_{k\to \infty} \sum_{\alpha\in \mathcal{A}} \lambda^{(k)}(\mathcal{Z}_{\alpha})
=\lim_{k\to \infty} \lambda^{(k)}\left( \cup_{\alpha\in A}\, \mathcal{Z}_{\alpha}\right)
\end{equation*}
and the conclusion follows.
\end{proof}

\begin{theorem}\label{thm:lambda(subtraction)}
Suppose that $\mathcal{Z}'\subseteq \mathcal{Z}\subseteq \mathbb{Z}^w$ such that $\lambda(\mathcal{Z}')$ and $\lambda(\mathcal{Z})$ both exist, then
\begin{equation*}
\lambda( \mathcal{Z}\setminus \mathcal{Z}') =\lambda(\mathcal{Z})-\lambda(\mathcal{Z}')\,.
\end{equation*}
In particular, for the complement $\mathcal{Z}^c$ of $\mathcal{Z}$ in $\mathbb{Z}^w$, we have
\begin{equation*}
\lambda( \mathcal{Z}^c) =1-\lambda(\mathcal{Z})\,.
\end{equation*}
\end{theorem}

\begin{proof}
By Definition \ref{def:lambda}, we have
\begin{equation*}
\lambda(\mathcal{Z})-\lambda(\mathcal{Z}')
=\lim_{k\to \infty} \lambda^{(k)}(\mathcal{Z})-\lim_{k\to \infty} \lambda^{(k)}(\mathcal{Z}')
=\lim_{k\to \infty} \left( \lambda^{(k)}(\mathcal{Z})-\lambda^{(k)}(\mathcal{Z}')\right)
=\lim_{k\to \infty} \lambda^{(k)}( \mathcal{Z}\setminus \mathcal{Z}')
\end{equation*}
and the conclusion follows.
\end{proof}

\begin{theorem}\label{thm:lambda(subset)=0}
Suppose that $\mathcal{Y}\in \mathbb{Z}^w$ such that $\lambda(\mathcal{Y})=0$\,, then for any $\mathcal{Z}\subseteq \mathcal{Y}$, we have $\lambda(\mathcal{Z})=0$ as well.
\end{theorem}
\begin{proof}
Since $\lambda^{(k)}(\mathcal{Z})\ge 0$\,, $\mathcal{Z}\subseteq \mathcal{Y}$ and $\lim_{k\to \infty} \lambda^{(k)}(\mathcal{Y})=\lambda(\mathcal{Y})=0$\,, we obtain
\begin{equation*}
0\le\liminf_{k\to \infty} \lambda^{(k)}(\mathcal{Z})\le \limsup_{k\to \infty} \lambda^{(k)}(\mathcal{Z})\le \limsup_{k\to \infty} \lambda^{(k)}(\mathcal{Y})
=\lambda(\mathcal{Y})=0\,.
\end{equation*}
Therefore
\begin{equation*}
\lim_{k\to \infty} \lambda^{(k)}(\mathcal{Z})=0\,,
\end{equation*}
as desired.
\end{proof}

\subsection{\texorpdfstring{Connection between $\lambda$ and $\lambda_{p^s}$}{Connection between lambda and lambda{\_}\{p{\^{}}s\}}}
\label{sec:lambda(Z)}
$ $

We show that the density $\lambda$ exists and in fact, equals the product of some $\lambda_{p^s}$'s.

\begin{asmp}\label{asmp:coprime}
For all $1\le i\le w$, we have $$\gcd(F_1,F_2,\dots,F_h)=\gcd\,(F:F\in U_i)=1\   \mathrm{in}\ \mathbb{Q}[x_1,x_2,\dots,x_d]\,.$$
\end{asmp}

\begin{theorem}\label{thm:lambda(Z)}
Suppose that Assumption \ref{asmp:coprime} holds.  
Given positive integers $r\le w$ and $y_i$, $1\le i\le r$, let $y=\prod_{j=1}^{\infty} p_j^{s_j}$ with $p_j$ the $j$-th smallest prime and $s_j$ nonnegative integers, $j=1,2,\dots$ such that $y_i\,|\,y$ for all $1\le i\le r$,
then the probability $\lambda(\mathcal{Z})$ exists for 
\begin{equation}\label{eq:Z}
\mathcal{Z}=\left\{(z_1,z_2,\dots,z_w) \in \mathbb{Z}_+^w:\ z_i=y_i\,,\  \forall\ i\le r\right\},
\end{equation} 
and in fact 
\begin{equation}\label{eq:lambda(Z)}
\lambda(\mathcal{Z}) = \prod_{j=1}^{\infty} \lambda_{p_j^{s_{j}+1}}(\mathcal{Z})\,.
\end{equation}
\end{theorem}

\begin{remark}\label{rmk:0}
{\rm
(i) The right-hand side of \eqref{eq:lambda(Z)} is well-defined since $\lambda_{p^s}(\cdot)\in [0,1]$ for all $p$ and $s$.

(ii) The special case that all $s_j$'s are either $0$ or $1$ follows from \cite[Theorem 2.3]{E}.

(iii) We have assumed that the $y_i$'s are positive. In fact, in the case that $y_i=0$ for some $i$, we have $\lambda(\mathcal{Z})=0$ on the strength of Theorem \ref{thm:lambda(subset)=0} and that the probability that a nonzero polynomial at a random integer vector equals zero is $0$ (see Theorem \ref{thm:G=0}\,(ii) below). }
\end{remark}

\smallskip
To prove Theorem \ref{thm:lambda(Z)}, we need Theorem \ref{thm:lambda_{P_P}} and the following two lemmas.

\begin{lemma}{\rm (\cite[Lemma 5.1]{P} or \cite[Lemma 21]{PS})}
\label{lemma:poonen}
Suppose that  $F,G\in \mathbb{Z}[x_1,x_2,\dots,x_d]$ are relatively prime as elements of $\mathbb{Q}[x_1,x_2,\dots,x_d]$. Let 
$\nu_{\ell}^{(k)}$ be the probability that $p\, | \,F(x), G(x)$ for some prime $p>\ell$  with $x$ uniformly distributed over $\mathbb{Z}_{(k)}^d$, i.e.,
$$\nu_{\ell}^{(k)}:=\# \left\{x\in \mathbb{Z}_{(k)}^d:\ \exists\ \mathrm{prime}\ p>\ell\ \,\mathrm{s.t.}\,\ p\, | \,F(x), G(x)\right\}/\,(2k+1)^d.$$
Then 
$$\lim_{\ell\to \infty} \limsup_{k\to \infty}\, \nu_{\ell}^{(k)}= 0\,.$$
\end{lemma}

\begin{lemma}\label{lemma:(G,H)=1}
Suppose that $G_1,\dots,G_h\in \mathbb{Q}[x_1,x_2,\dots,x_d]$ $(h\ge 2)$ are relatively prime, then there exists  $v=(v_3,\dots,v_{h})\in \mathbb{Z}^{h-2}$ such that
$$\gcd\left(G_1, G_2+\sum_{i=3}^{h} v_i G_i\right) =1.$$
\end{lemma}

\begin{proof} 
We prove by induction on $h$. The case $h=2$ is trivial since $\gcd(G_1,G_2)=1$.

\smallskip
\noindent
{\it Base case}: $h=3$\,.

We prove by contradiction. 
Assume the contrary that
$$\gcd(G_1,G_2+z\, G_3 ) \neq 1, \quad \forall\ z\in \mathbb{Z}\,.$$

Suppose that the polynomial factorization of $G_1$ 
is $\phi_{1} \phi_{2}\cdots \phi_{u}$\,, then each $G_2+z\, G_3$ is a multiple of some factor $\phi_{u(z)}$ of $G_1$ $(1\le u(z)\le u)$. Since there are infinitely many $z$'s, by the pigeonhole principle, at least two of the $u(z)$'s are the same, say $u(z)=u(z')\,(z\neq z')$. Then
$$\phi_{u(z)}\,|\,( G_2+z\, G_3) -( G_2+z' G_3) = (z-z')\,G_3\,$$
thus $\phi_{u(z)}\,|\,G_3$ and hence $\phi_{u(z)}\,|\,( G_2+z\, G_3) -z\,G_3=G_2$\,. Recall that $\phi_{u(z)}\,|\,G_1$ as well. This contradicts with the condition that $G_1,G_2$ and $G_3$ are relatively prime.

\smallskip
\noindent
{\it Inductive step}: from $h-1$ to $h\,(\ge 4)$. Assume that the statement holds for $h-1$.

Let $H:=(G_2,G_3,\dots,G_{h}) $ and $H_i:=G_i/H\,(2\le i\le h)$, then 
\begin{equation}\label{eq:(G_1,H)}
\gcd(G_1,H) =\gcd(G_1,G_2,\cdots,G_h)=1=\gcd( H_2,H_3,\cdots,H_h) \,.
\end{equation}

According to the induction hypothesis for $H_2,H_3,\cdots,H_{h}$, there exists $v=(v_4,\dots,v_{h})\in\mathbb{Z}^{h-3}$  such that
\begin{equation}\label{eq:H'_3}
H'_3:=H_3+\sum_{i=4}^h v_i H_i
\end{equation}
satisfies
$$\gcd(H_2,H'_3 )=1. $$
Combining with \eqref{eq:(G_1,H)} gives
$$\gcd(G_1,G_2,H'_3H) =\gcd(G_1,H_2H,H'_3H)=\gcd( G_1,\gcd(H_2H,H'_3H ) )= \gcd(G_1,H)=1. $$
Thus we can apply the base case $h=3$ to $G_1,G_2,H'_3H$
to get an integer $z$ such that 
$$\gcd(G_1, G_2+z H'_3H) = 1.$$

Finally, we represent $H'_3H$ back to a linear combination of the $G_i$'s with integer coefficients by definition \eqref{eq:H'_3}:
$$H'_3H=H_3H+\sum_{i=4}^h v_i H_iH=G_3+\sum_{i=4}^h v_i G_i\,,$$
therefore
$$\gcd\left( G_1, G_2+z\,G_3+z \sum_{i=4}^h v_i G_i\right)  = 1,$$
namely, the statement holds for $h$ with the new $v=(z,z v_4,\dots,z v_h)$.
\end{proof} 

Now we are ready to prove Theorem \ref{thm:lambda(Z)}.

\begin{proof}[Proof of Theorem \ref{thm:lambda(Z)}]
Let
$$\mathcal{P}_{\ell}:=
\{(p_j,s_j+1)\}_{j=1}^{\ell}\,,\quad \ell\in \mathbb{Z}_+\,,$$
then Theorem \ref{thm:lambda_{P_P}} gives
\begin{equation*}
\lambda_{P_{\mathcal{P}_{\ell}}}(\mathcal{Z})
=\prod_{j=1}^{\ell} \lambda_{p_j^{s_{j}+1}}(\mathcal{Z})\,.
\end{equation*}
Since $\lambda_{p_j^{s_{j}+1}}(\mathcal{Z})\in [0,1]$ for all $j$, we can let $\ell\to \infty$\,:
\begin{equation}\label{eq:lim lambda_{P_l}}
\lim_{\ell\to \infty}\lambda_{P_{\mathcal{P}_{\ell}}}(\mathcal{Z})
=\prod_{j=1}^{\infty} \lambda_{p_j^{s_{j}+1}}(\mathcal{Z})=\mathrm{RHS\ of}\ \eqref{eq:lambda(Z)}\,.
\end{equation}
Therefore it suffices to show that
\begin{equation}\label{eq:lambda^k}
\lim_{k\to \infty}\lambda^{(k)}(\mathcal{Z})
=\lim_{\ell\to \infty}\lambda_{P_{\mathcal{P}_{\ell}}}(\mathcal{Z})\,.
\end{equation}

Since $y$ is finite, there exists $j^*\in \mathbb{Z}_+$ such that $s_j=0$ for all $j>j^*$. Let 
$$\mathcal{I}=\left\{(z_1,z_2,\dots,z_w)\in \mathbb{Z}_+^w: z_1=z_2=\cdots=z_r=1\right\}$$ 
then for any $j>j^*$, the sets $\mathcal{Z}$ and $\mathcal{I}$ are equivalent mod $p_j$ under multiplication of the components by units.  

We define  for $\ell>j^*$,
\begin{equation*}
A(\ell):=\left\{ x\in \mathbb{Z}^{d}: g(x)\in\mathcal{Z}\ (\mathrm{mod}\ P_{\mathcal{P}_{\ell}})\right\},\quad 
A^{(k)}(\ell):=\left\{ x\in \mathbb{Z}_{(k)}^{d}: g(x)\in\mathcal{Z}\ (\mathrm{mod}\ P_{\mathcal{P}_{\ell}}) \right\},
\end{equation*}
\begin{equation*}
A^{(k)}:=\left\{ x\in \mathbb{Z}_{(k)}^{d}: g(x)\in\mathcal{Z}\right\}\,\left(\subseteq A^{(k)}(\ell)\right) \,,
\end{equation*}
and
\begin{equation*}
B^{(k)}(\ell):=A^{(k)}(\ell)\setminus A^{(k)},
\end{equation*}
then 
\begin{equation}\label{eq:A^{(k)}}
\lambda^{(k)}(\mathcal{Z})
=\frac{\# A^{(k)}}{K^{d}}\,,\quad \lambda_{P_{\mathcal{P}_{\ell}}}^{(k)}(\mathcal{Z})
=\frac{\#A^{(k)}(\ell)} {K^{d}}
=\frac{\# A^{(k)}+\# B^{(k)}(\ell)}{K^{d}}
\end{equation}
with $K:=2k+1$. Therefore
\begin{equation}\label{eq:lambda_{P_{P_l}}}
\lambda_{P_{\mathcal{P}_{\ell}}}(\mathcal{Z})
=\lim_{k\to \infty} \lambda_{P_{\mathcal{P}_{\ell}}}^{(k)}(\mathcal{Z})
=\lim_{k\to \infty}  \frac{\# A^{(k)}+\# B^{(k)}(\ell)}{K^{d}}\,.
\end{equation}
Combining with the first equation in \eqref{eq:A^{(k)}} leads to 
\begin{equation*}
\limsup_{k\to \infty} \lambda^{(k)}(\mathcal{Z})
\le 
\limsup_{k\to \infty}  \frac{\# A^{(k)}+\# B^{(k)}(\ell)}{K^{d}} 
=
\lambda_{P_{\mathcal{P}_{\ell}}}(\mathcal{Z}) 
\end{equation*}
and
\begin{equation*}
\liminf_{k\to \infty}\lambda^{(k)}(\mathcal{Z})
\ge 
\liminf_{k\to \infty}  \frac{\# A^{(k)}+\# B^{(k)}(\ell)}{K^{d}} - \limsup_{k\to \infty}  \frac{\# B^{(k)}(\ell)}{K^{d}}
\ge
\lambda_{P_{\mathcal{P}_{\ell}}}(\mathcal{Z}) -
\limsup_{k\to \infty}  \frac{\# B^{(k)}(\ell)}{K^{d}}\,.
\end{equation*}
Once we show that
\begin{equation}\label{eq:limB=0}
\lim_{\ell\to \infty}\limsup_{k\to \infty}  \frac{\# B^{(k)}(\ell)}{K^{d}}=0\,,
\end{equation} 
taking $\ell\to \infty$ in the above two inequalities will yield \eqref{eq:lambda^k}.

\smallskip
Now let us prove \eqref{eq:limB=0}. For any $x\in B^{(k)}(\ell)$\,, there exists $j>\ell\,(>j^*)$ such that $g(x)\notin\mathcal{I}$  (mod $p_j^{s_j+1}=p_j$) (recall that $\mathcal{Z}$ and $\mathcal{I}$ are equivalent). Hence $p_j \, | \,g_\eta(x)$ for some $\eta\le r$. 

Recall that $g_{\eta}$ is the gcd of some relatively prime $F_i$'s. If two or more $F_i$'s are involved, then applying Lemma \ref{lemma:(G,H)=1} to these $F_i$'s leads to two relatively prime linear combinations $\mathcal{G}_{\eta}$ and $\mathcal{H}_{\eta}$ of these $F_i$'s with integer coefficients. If there is only one $F_i$ involved, then it must be a constant since the gcd of itself is $1$ in $\mathbb{Q}[x_1,x_2,\dots,x_d]$. In this case, we take $\mathcal{G}_{\eta}=\mathcal{H}_{\eta}=F_i$ so that  $\gcd(\mathcal{G}_{\eta},\mathcal{H}_{\eta})=1$ still holds. 

Since $p_j \, | \,g_{\eta}(x)$, we have $p_j \, \big| \, \mathcal{G}_{\eta}(x),\mathcal{H}_{\eta}(x)$. Hence
\begin{equation}\label{eq:B_t^k}
B^{(k)}(\ell)\subseteq \bigcup_{{\eta}=1}^r\, \overline{B}_{\eta}^{(k)}(\ell)\,,
\end{equation}
where
\begin{equation*}
\overline{B}_{\eta}^{(k)}(\ell):=\left\{x\in \mathbb{Z}_{(k)}^{d}: \exists\ j>\ell \ \ \mathrm{s.t.}\ \ p_j \, \big| \, \mathcal{G}_{\eta}(x)\,,\mathcal{H}_{\eta}(x)\right\}\,.
\end{equation*}

Applying Lemma \ref{lemma:poonen} to $\mathcal{G}_{\eta}$ and $\mathcal{H}_{\eta}$ gives
\begin{equation*}
\lim_{\ell\to \infty}\limsup_{k\to \infty}  \frac{\#\overline{B}_{\eta}^{(k)}(\ell)} {K^{d}}=0\,,\ \ \forall\ {\eta}\,.
\end{equation*}
Combining with \eqref{eq:B_t^k}, we obtain 
\begin{equation*}
\limsup_{\ell\to \infty}\limsup_{k\to \infty}  \frac{\# B^{(k)}(\ell)}{K^{d}}
\le \limsup_{\ell\to \infty}\limsup_{k\to \infty} \sum_{{\eta}=1}^r \frac{\#\overline{B}_{\eta}^{(k)}(\ell)} {K^{d}}
\le\sum_{{\eta}=1}^r \limsup_{\ell\to \infty}\limsup_{k\to \infty}  \frac{\#\overline{B}_{\eta}^{(k)}(\ell)} {K^{d}}
=0\,.
\end{equation*} 
On the other hand, since $\# B^{(k)}(\ell)\ge 0$\,, we have
\begin{equation*}
\liminf_{\ell\to \infty}\limsup_{k\to \infty}  \frac{\# B^{(k)}(\ell)}{K^{d}}
\ge 0\,.
\end{equation*} 
Hence \eqref{eq:limB=0} indeed holds.
\end{proof}

\subsection{Relatively Prime Polynomial Values}
$ $

An interesting application of Theorem \ref{thm:lambda(Z)} is to determine the probability that the polynomial values are relatively prime. 

\begin{theorem}\label{thm:coprime}
Let\, $w=1$ and\, $U_1=\{\{F_1,F_2,\dots,F_h\}\}$ in Definition \ref{def:lambda}. 

\noindent
{\rm (a)}
If $F_1,F_2,\dots,F_h$ are not relatively prime in $\mathbb{Q}[x_1,x_2,\dots,x_d]$, then $\lambda(\{1\})=0$\,;

\noindent
{\rm (b)} 
If $F_1,F_2,\dots,F_h$ are relatively prime, i.e.,
\begin{equation}\label{eq:coprime}
\gcd(F_1,F_2,\dots,F_h)=1\ \ \mathrm{in}\   \mathbb{Q}[x_1,x_2,\dots,x_d]\,. 
\end{equation}
then we have

\noindent
{\rm (i)} $\lambda(\{1\})$ exists and 
$$\lambda(\{1\})=\prod_{p}\lambda_{p}(\{1\})\,;$$

\noindent
{\rm (ii)} the asymptotic result
\begin{equation}\label{eq:p^{-2}}
\lambda_{p}(\{0\})=O(p^{-2})\,;
\end{equation}

\noindent
{\rm (iii)}
$\lambda(\{1\})=0$
if and only if $\lambda_{p}(\{1\})=0$ for some prime $p$, i.e., if and only if there exists a prime $p$ such that $F_1(x),F_2(x),\dots,F_h(x)$ are  multiples of $p$ for all $x$\,;
in words, the probability that the values of relatively prime polynomials at a random integer are relatively prime is $0$ if and only if there exists a prime $p$ such that these polynomials are all always  multiples of $p$.
\end{theorem}

\begin{remark} 
Theorem \ref{thm:coprime}\,(b)(ii) and Lemma \ref{lemma:sigma} in the
proof below are special cases of the Lang-Weil bound \cite[Theorem
  1]{LW}. We present a considerably simpler and more approachable
proof.  
As mentioned  in Remark \ref{rmk:0} and \cite[Remark of Lemma 21]{PS},
Theorem \ref{thm:coprime}\,(b)(i) follows from \cite[Theorem 2.3]{E};
whereas its special case $h=2$ was shown in \cite[Theorem 3.1]{P}. 
\end{remark}

\begin{proof}
(a) Let $G=\gcd(F_1,F_2,\dots,F_h)$, then $G$ is a non-constant polynomial. If the gcd $g(x)=1$, then $G(x)=\pm 1$. Thus
$\lambda^{(k)}(\{1\})\le \sigma^{(k)}_{G=1}+\sigma^{(k)}_{G=-1}\to 0$ as $k\to \infty$ on the strength of Theorem \ref{thm:G=0}\,(ii), 
where $\sigma^{(k)}_{G=c}$ $(c=\pm 1)$ is the probability that $G(x)=c$ with $x$ uniformly distributed over $\mathbb{Z}_{(k)}^d$.
Hence $\lambda(\{1\})=0$\,.

\smallskip
(b) (i) follows directly from Theorem \ref{thm:lambda(Z)}. 
For (ii), we prove by induction on $d$. 
First, we notice the following facts:
\smallskip

\noindent
{\it 1}. If $h=1$, then $F_1$ must be a constant due to Assumption \eqref{eq:coprime}. Hence $\lambda_{p}(\{0\})=0$ for all $p>|F_1|$ and \eqref{eq:p^{-2}} follows.
\smallskip

\noindent
{\it 2}. If $h\ge 2$\,, by Lemma \ref{lemma:(G,H)=1}, there exist two linear combinations $G$ and $H$ of the $F_i$'s with integer coefficients such that $\gcd(G,H)=1$ in $\mathbb{Q}[x_1,x_2,\dots,x_d]$. Then $p\,|\,g(x)$ implies that $p\,|\,\gcd(G(x),H(x))$, so it suffices to prove for the case $h=2$\,.  

\smallskip
\noindent
{\it 3}. Assume that $h=2$\,. Let $L$ be the greatest total degree of the $F_i$'s. If $L=0$\,, then $F_1,F_2$ and thus $g$ are nonzero constants. Thus $\sigma_{p}=0$ for any  $p>|g|$ and \eqref{eq:p^{-2}} follows, so we only need to prove for $L\ge 1$. 

\smallskip
\noindent
{\it Base case}: $d=1$. Assume that $h=2$\,.

Thanks to Assumption \eqref{eq:coprime}, there exist $H_1,H_2\in \mathbb{Z}[x_1]$ such that 
$H_1F_1+H_2F_2=C$ with $C$ a positive integer constant. If $p\,|\,g(x)$, then $p\,|\,C$ as well. Hence 
$\sigma_{p}=0$ for all  $p>C$ and \eqref{eq:p^{-2}} follows.

\smallskip
\noindent
{\it Inductive step}: from $d-1$ to $d\,(\ge 2)$. Assume that the statement holds for $d-1$ and that $h=2$ and $L\ge 1$.

Since $L\ge 1$, without loss of generality, we can assume that $F_1$ is not a constant and $x_1$ appears in $F_1$\,.
We recast $F_i$ as a univariate polynomial $G_i\in (\mathbb{Z}[x_2,\dots,x_d])[x_1]$ of degree $L_i$\,, $i=1,2$, then $L_1\ge 1$. Let $\gamma_{i,j}\in \mathbb{Z}[x_2,\dots,x_d]\,(0\le j\le L_i)$ be the coefficients of $x_1^j$ in $G_i$\,.

Since $F_1$ and $F_2$ are relatively prime in $\mathbb{Q}[x_1,x_2,\dots,x_d]$ by  Assumption   \eqref{eq:coprime}, we have 
\begin{equation*}
\gcd\left(\gamma_{i,j}:1\le i\le 2\,, 0\le j\le L_i \right) =1=\gcd(G_1,G_2)\ \ \mathrm{in}\  (\mathbb{Q}[x_2,\dots,x_d])[x_1]\,.
\end{equation*}
As a result, there exist $H_1,H_2\in (\mathbb{Z}[x_2,\dots,x_d])[x_1]$ and $H_3\in \mathbb{Z}[x_2,\dots,x_d]$ such that 
$H_1G_1+H_2G_2=H_3$ and $(H_1,H_2,H_3)=1$ in $\mathbb{Q}[x_1,x_2,\dots,x_d]$.

If $p\,|\,g(x)$, then $p\,|\left( G_i(x_2,\dots,x_d)\right)  (x_1) \,(\forall\ i),H_3(x_2,\dots,x_d)$ and either 

\noindent
(1) $p\,|\,\gamma_{i,j}(x_2,\dots,x_d)$ for all $i$ and $j$; or 

\noindent
(2) $p\nmid \gamma_{i,j}(x_2,\dots,x_d)$ for some $i,j$. 

\smallskip
{\it Case} (1). 
Recall that $L_1\ge 1$. By the induction hypothesis for the at least two polynomials: $\gamma_{i,j}(x_2,$ $\dots,$\,$x_d)\,(1\le i\le 2\,, 0\le j\le L_i)$, the probability that Case (1) happens with $(x_2,\dots,x_d)$ uniformly distributed on $(\mathbb{Z}/p\mathbb{Z})^{d-1}$ is $O(p^{-2})$.

\smallskip
{\it Case} (2). 
We need the following asymptotic result.
\begin{lemma}\label{lemma:sigma}
Let $G\in \mathbb{Z}[x_1,x_2,\dots,x_d]$ be a nonzero polynomial, $p$ a  prime, and $\sigma_{p}$ the probability that $p\,|\,G(x)$ with $x$ uniformly distributed over $(\mathbb{Z}/p\mathbb{Z})^d$, then we have
\begin{equation}\label{eq:sigma_p}
\sigma_{p}=O(p^{-1})\,.
\end{equation}  
\end{lemma}
\begin{proof}
Let $L$ be the total degree of $G$. If $L=0$\,, then $G$ is a nonzero constant. For any prime $p>G$, we have $\sigma_{p}=0$\,, thus \eqref{eq:sigma_p} holds. 

Now we assume that $L\ge 1$. We prove by induction on $d$. 

\smallskip
\noindent
{\it Base case}: $d=1$.

Since the number of roots of $G$ mod $p$ is at most $L$, we get $\sigma_{p}\le L/p$ and hence \eqref{eq:sigma_p}. 

\smallskip
\noindent
{\it Inductive step}: from $d-1$ to $d\,(\ge 2)$. Assume that the statement holds for $d-1$.

We recast $G$ as a univariate polynomial $G_1\in (\mathbb{Z}[x_2,x_3,\dots,x_d])[x_1]$. Let $\gamma_1\in \mathbb{Z}[x_2\dots,x_d]$ be the leading coefficient of $G_1$.
Observe that the total degree of $G_1$ is at most $L$\,. If $\gamma_1(x_2,\dots,x_d)\not\equiv 0$ (mod $p$), then the probability that $p\,|\,G_1(x_1)$ with $x_1$ uniformly distributed over $\mathbb{Z}/p\mathbb{Z}$ is no greater than $L/p$\,, according to the base case $d=1$.
On the other hand,  the probability that $p\,|\,\gamma_1(x_2,\dots,x_d)$ with $(x_2,\dots,x_d)$ uniformly distributed over $(\mathbb{Z}/p\mathbb{Z})^{d-1}$ is $O(p^{-1})$ by the induction hypothesis for $\gamma_1$\,.
Combining these two cases, we conclude that the probability that $p\,|\,G(x_1,x_2,\dots,$ $x_d)$ with $(x_1,x_2,\dots,x_d)$ uniformly distributed over $(\mathbb{Z}/p\mathbb{Z})^{d}$ is at most
$L/p+O(p^{-1})=O(p^{-1})$. Hence the statement holds for $d$, as desired.
\end{proof}

Now we go back to the proof of Theorem \ref{thm:coprime}\,(b)(ii).
Thanks to Lemma \ref{lemma:sigma}, the probability that $p\,|\,H_3(x_2,\dots,x_d)$  with $(x_2,\dots,x_d)$ uniformly distributed on $(\mathbb{Z}/p\mathbb{Z})^{d-1}$ is $O(p^{-1})$; moreover, for each $(x_2,\dots,x_d)$ that satisfies $p\nmid \gamma_{i,j}(x_2,\dots,x_d)$ for some $i,j$, the probability that $p\,|\left( G_i(x_2,\dots,x_d)\right) (x_1)$ with $x_1$ uniformly distributed on $\mathbb{Z}/p\mathbb{Z}$ is $O(p^{-1})$. 
Hence the probability that Case (2) happens with $(x_1,x_2,$ $\dots,$\,$x_d)$ uniformly distributed on $(\mathbb{Z}/p\mathbb{Z})^{d}$ is $( O(p^{-1})) ^2=O(p^{-2})$. 
  
\smallskip
Combining Cases (1) and (2), we conclude that the statement holds for $d$ as well, as desired.

\smallskip
(iii) If $\lambda_{p}(\{1\})=0$ for some prime $p$, then  $\lambda(\{1\})=0$ by (i).

Now assume that $\lambda_{p}(\{1\})>0$ for all prime $p$. 
On the strength of (ii), 
there exist a positive constant $c$  and a positive integer $j^*$ such that 
$$p_{j^*}>1+c>\sqrt{c}\,,\quad \lambda_{p_j}(\{0\})\le c\,p_j^{-2}\,,\quad \forall\ j\ge j^*.$$
Thus
\begin{align*}
&\ \prod_{j=j^*}^{\infty}\left(1- \lambda_{p_j}(\{0\})\right) 
\ge \prod_{j=j^*}^{\infty}\left(1- \frac{c}{p_j^2}\right) 
\ge 1-\sum_{j=j^*}^{\infty} \frac{c}{p_j^2}
\ge 1-\sum_{i=p_{j^*}}^{\infty} \frac{c}{i^2}\\
\ge&\ 1-\sum_{i=p_{j^*}}^{\infty} c\left( \frac{1}{i-1}-\frac{1}{i}\right) 
=1-\frac{c}{p_{j^*}-1}>0\,,
\end{align*}
where in the second inequality, we take advantage of the well-known inequality:
\begin{equation}\label{eq:1-delta}
(1-\delta_1)(1-\delta_2)\cdots(1-\delta_{u})\ge 1-\delta_1-\delta_2-\cdots-\delta_{u}
\end{equation} 
for $\delta_1,\delta_2,\dots,\delta_{u}\in [0,1]$, which can be proved easily by induction on $u$ (base cases: $u=1,2$; inductive step from $u$ to $u+1$: $(1-\delta_1)(1-\delta_2)\cdots(1-\delta_{u+1})\ge (1-\delta_1)(1-\delta_2-\cdots-\delta_{u+1}) \ge 1-\delta_1-\delta_2-\cdots-\delta_{u+1} $).

Hence
$$\prod_{p}\lambda_{p}(\{1\})
=\prod_{j=1}^{j^*-1}\lambda_{p_j}(\{1\})\cdot \prod_{j=j^*}^{\infty}\left(1- \lambda_{p_j}(\{0\})\right)>0\,.$$
\end{proof}

\subsection{Zero Polynomial Values}\label{sec:g_i=0}
$ $

Remark \ref{rmk:0}\,(iii) used a well-known result that the probability that a nonzero polynomial at a random integer vector equals zero is $0$ (\cite[Lemma 4.1]{P}). 
We conclude this section with a different proof by estimating this probability from above by $\sigma_p$ and applying Lemma \ref{lemma:sigma}.

\begin{theorem}\label{thm:G=0}
Let $G\in \mathbb{Z}[x_1,x_2,\dots,x_d]$ be a nonzero polynomial, $p$ a prime, $\sigma_{p}^{(k)}$ the probability that $p\,|\,G(x)$ with $x$ uniformly distributed over $\mathbb{Z}_{(k)}^d$, and $\sigma_{p}$ the 
probability that $p\,|\,G(x)$ with $x$ uniformly distributed over $(\mathbb{Z}/p\mathbb{Z})^d$, then 

\noindent
{\rm (i)} we have
$$\sigma_{p}^{(k)}\to \sigma_{p}\ \  \mathrm{as}\ \ k\to \infty\,,\quad \mathrm{and}\quad \sigma_{p}^{(k)}\le 2^d\sigma_{p}\,, \ \forall \ k>(p-1)/2\,;$$

\noindent
{\rm (ii)} 
the probability $\sigma^{(k)}$ that $G(x)=0$ with $x$ uniformly distributed over $\mathbb{Z}_{(k)}^d$ goes to $0$ as $k\to \infty$\,;
in words, the probability that a nonzero polynomial at a random integer vector equals zero is $0$. As a consequence, for any given integer $c$, the probability that $G(x)=c$ is either $0$ or $1$ (consider the polynomial $G(x)-c$).
\end{theorem}

\begin{proof}
(i) We follow a similar approach as in the proof of the first equality of \eqref{eq:lambda_{P_P}(Z)}. 
Let $k\in \mathbb{Z}$ such that $K:=2k+1>p$.
Then there exists $q\in \mathbb{Z}_+$ such that
\begin{equation}\label{eq:q'}
q\cdot P_{\mathcal{P}} \le  K < (q+1)\cdot P_{\mathcal{P}}\,. 
\end{equation}  
It follows that  for any integer $y$, there are either $q$ or $q+1$ numbers among $\mathbb{Z}_{(k)}$ that equal $y$ mod $p$.
Further, the number of $x\in (\mathbb{Z}/p\mathbb{Z})^d$ for which $p\,|\,G(x)$ is $p^d\sigma_{p}$\,, thus the number of $x\in \mathbb{Z}_{(k)}^d$ for which $p\,|\,G(x)$ is between 
$q^{d}p^d \sigma_{p}$ and $(q+1)^{d} p^d \sigma_{p}$\,. Therefore
\begin{equation}\label{eq:sigma_p^k}
\sigma_{p}^{(k)}\in \left[\frac{q^{d} p^d \sigma_{p}}{ K^{d} }\,, \frac{(q+1)^{d} p^d \sigma_{p}}{ K^{d} }\right]:=\Sigma_k\,.
\end{equation}
Thanks to \eqref{eq:q'}, we have
\begin{equation}\label{eq:Sigma_k}
\Sigma_k \subseteq \left[\frac{q^{d} p^d \sigma_{p}}{ [ (q+1)p]^d }\,, \frac{(q+1)^{d} p^d \sigma_{p}}{ ( qp) ^d }\right] 
=\left[ \left(\frac{q}{q+1}\right) ^{d} \sigma_{p}\,, \left(\frac{q+1}{q}\right) ^{d} \sigma_{p} \right] ,
\end{equation}
whose left and right endpoints both converge to $\sigma_{p}$ as $q\to \infty$\,. Hence  
\begin{equation*}
\sigma_{p}^{(k)}\to \sigma_{p}\,,\quad \mathrm{as}\ q\to \infty\,,\ \mathrm{or\ equivalently,\ as}\ k\to \infty\,. 
\end{equation*} 

Additionally, we deduce $\sigma_{p}^{(k)}\le 2^d\sigma_{p}$ from \eqref{eq:sigma_p^k} and \eqref{eq:Sigma_k} along with $q\ge 1$.

\smallskip
(ii) The probability $\sigma^{(k)}$ is no greater than $\sigma_{p}^{(k)}$, which by virtue of (i) and Lemma \ref{lemma:sigma}, converges to $0$ as $p,k\to \infty$ with $k>(p-1)/2$\,. 
\end{proof}

\section{SNF Distribution}\label{sec:SNF}
Let $m\le n$ be two positive integers. 
We shall define the {\it density} of SNF of a random $n\times m$ integer matrix as the limit (if exists) of the density of SNF of a random $n\times m$ matrix with entries independent and uniformly distributed over $\mathbb{Z}_{(k)}$ as $k\to \infty$ (see Definition \ref{def:mu} below for a precise definition). 

If we regard the minors of an $n\times m$ matrix  as polynomials of the $nm$ matrix entries with integer coefficients, then the SNF of a matrix is uniquely determined by the values of these polynomials. 
Specifically, let $x_1,x_2,\dots,x_{nm}$ be the $nm$ entries of an $n\times m$ matrix, $F_j$'s be the minors of an $n\times m$ matrix as elements in $\mathbb{Z}[x_1,x_2,\dots,x_{nm}]$, $U_i$ be the set of $i\times i$ minors $(1\le i\le m)$, then the SNF of this matrix is the diagonal matrix whose $i$-th diagonal entry is $0$ if $g_i(x)=0$ and $g_i(x)/g_{i-1}(x)$ otherwise, where $x=(x_1,x_2,\dots,x_{nm})$ and $g_i(x)$ is defined in \eqref{eq:gi}. 

In this spirit, the multi-gcd distribution as well as the results in Sections \ref{sec:lambda_{P_P}}--\ref{sec:lambda(Z)}  have analogues for the SNF distribution of a random integer matrix. 
This section presents these analogues and the next section will use them to compute the density $\mu$ for some interesting types of sets.

Conventionally, the SNF is only defined for a nonzero matrix; however, for convenience, we shall define the SNF of a zero matrix to be itself, so that SNF is well-defined for all matrices.
This definition does not change the density (if exists) of SNF of a random $n\times m$ integer matrix since 
the probability of a zero matrix with entries from $\mathbb{Z}_{(k)}$ is  $1/(2k+1)^{nm}$, which converges to $0$ as $k\to \infty$\,. 

We denote the SNF of an $n\times m$ matrix $M$ by $\mathrm{SNF}(M)=(\mathrm{SNF}(M)_{i,j})_{n\times m}$ and let $\mathbb{S}$ be the set of all candidates for SNF of an $n\times m$ integer matrix, i.e., the set of $n\times m$ diagonal matrices whose diagonal entries $(d_1,d_2,\dots,d_m)$ are nonnegative integers  such that $d_{i+1}$ is a multiple of $d_i$\,, $i=1,2,\dots, m-1$. 

For ease of notation, we shall always assume that the matrix entries are {\it independent and uniformly distributed} on its range (if known, e.g., $\mathbb{Z}_{(k)}$ or  $\mathbb{Z}/P_{\mathcal{P}} \mathbb{Z}$), and that the notion $\mathrm{SNF}(M)\in \mathcal{S}$ or $\mathrm{SNF}(M)=D$ (mod $P_{\mathcal{P}}$) for some $\mathcal{S}\subseteq \mathbb{S}$, $D\in \mathbb{S}$ and $P_{\mathcal{P}}=\prod_{(p,s)\in \mathcal{P}}p^s\in \mathbb{Z}_+$ implies the {\it equivalence of multiplication of the entries of $M$ by units} in $\mathbb{Z}/P_{\mathcal{P}} \mathbb{Z}$\,, 
thus we can assume for convenience that the entries of $\mathrm{SNF}(M)$ (mod $P_{\mathcal{P}}$) are zero or divisors of $P_{\mathcal{P}}$.

\begin{definition}\label{def:mu}
(i) For $\mathcal{S}\subseteq \mathbb{S}$\,, 
we denote by $\mu^{(k)}(\mathcal{S})$ the probability that $\mathrm{SNF}(M)\in \mathcal{S}$ with entries of $M$ from $\mathbb{Z}_{(k)}$. 
If \,$\lim_{k\to \infty} \mu^{(k)}(\mathcal{S})=\mu(\mathcal{S})$ exists, then we say that the {\em probability that $\mathrm{SNF}(M)\in \mathcal{S}$ with $M$ a random $n\times m$ integer matrix} is $\mu(\mathcal{S})$. If this is the case, then $\mu(\mathcal{S})\in [0,1]$ since  $\mu^{(k)}(\mathcal{S})\in [0,1]$ for all $k$.

\smallskip
(ii) We define similarly the {\rm SNF distribution over the ring of integers mod $p^s$}: for prime $p$ and positive integer $s$, we denote by $\mu^{(k)}_{p^s}(\mathcal{S})$  the  probability that the $\mathrm{SNF}(M)\in \mathcal{S}\ (\mathrm{mod}\ p^s)$ with entries of $M$ from $\mathbb{Z}_{(k)}$, and by $\mu_{p^s}(\mathcal{S})$ the probability that $\mathrm{SNF}(M)\in \mathcal{S}\ (\mathrm{mod}\ p^s)$ with entries of $M$ from $\mathbb{Z}/p^s \mathbb{Z}$\,.  

\smallskip
More generally, for a finite set \,$\mathcal{P}$ of prime and positive integer pairs $(p,s)$ (with $p$ a prime  and  $s$ a positive integer), we denote by $\mu^{(k)}_{P_{\mathcal{P}}}(\mathcal{S})$ the probability that \,$\mathrm{SNF}(M)\in \mathcal{S}\ (\mathrm{mod}\ P_{\mathcal{P}})$ with entries of $M$ from $\mathbb{Z}_{(k)}$, and by $\mu_{P_{\mathcal{P}} }(\mathcal{S})$ the probability that\, $\mathrm{SNF}(M)\in \mathcal{S}\ (\mathrm{mod}\ P_{\mathcal{P}})$ with entries of $M$ from $\mathbb{Z}/P_{\mathcal{P}}  \mathbb{Z}$\,.  Note that $\mu_{P_{\mathcal{P}}}(\mathcal{S})$ is the number of matrices $M$ over $P_{\mathcal{P}}$ such that $\mathrm{SNF}(M)\in \mathcal{S}\, (\mathrm{mod}\ P_{\mathcal{P}})$ divided by $P_{\mathcal{P}}^{nm}$.  
The situation discussed in the previous paragraph is the special case that \,$\mathcal{P}$ consists of only one element $(p,s)$ and $P_\mathcal{P}=p^s$.
\end{definition}

In this section, we establish
a formula for $\mu_{p^s}$, discuss
the properties of $\mu_{P_{\mathcal{P}}}$ and $\mu$, show the existence of $\mu$ and represent it as a product of $\mu_{p^s}$'s. 

\subsection{\texorpdfstring{SNF Distribution over $\mathbb{Z}/P_{\mathcal{P}} \mathbb{Z}\,$}{SNF Distribution over Z/(P{\_}P)Z}}
$ $

We have the following formula for $\mu_{p^s}$ and analogue of Theorem \ref{thm:lambda_{P_P}} for SNFs.

\begin{theorem}\label{thm:p^s}
{\rm{(i)}} Given a prime $p$, a positive integer $s$ and a sequence of integers $0=a_0\le a_1\le \cdots \le a_s \le a_{s+1} = m$\,, let ${\bm a}:=(a_1,a_2,\dots,a_s)$ and $D_{\bm a}\in \mathbb{S}$ be the diagonal matrix  with exactly $(a_i - a_{i-1})$ $p^{i-1}$'s, i.e., $a_i$ non-$p^i$-multiples,  $1\le i\le s$ on its diagonal. Then we have
\begin{equation}\label{eq:mu_{p^s}(D)}
\mu_{p^s}(\{D_{\bm a}\})=p^{- \sum_{i=1}^{s}(n-a_i)(m-a_i)} \cdot 
\frac{[p,n][p,m]}
{[p,n-a_s][p,m-a_s] \prod_{i=1}^{s} [p,a_i-a_{i-1}]}\,,
\end{equation} 
where  
\begin{equation*}
[p,0]=1,\quad [p,\ell]: = \prod_{j=1}^{\ell} \left( 1-p^{-j}\right) \,,\quad \ell\in \mathbb{Z}_+\,.
\end{equation*} 

{\rm{(ii)}}
For any $\mathcal{S}\subseteq  \mathbb{S}$, we have
\begin{equation*}
\mu_{P_{\mathcal{P}}}(\mathcal{S})
=\sum_{D\in \mathcal{S} \,(\mathrm{mod}\, P_{\mathcal{P}})} \mu_{P_{\mathcal{P}}}(\{D\}) 
\end{equation*}
and
\begin{equation}\label{eq:mu_{P_P}}
\lim_{k\to \infty}\mu^{(k)}_{P_{\mathcal{P}}}
=\mu_{P_{\mathcal{P}}}(\mathcal{S})
=\prod_{(p,s)\in \mathcal{P}}\mu_{p^s}(\mathcal{S}) \,.
\end{equation}
\end{theorem}

\begin{proof} 
(ii) and (iii) are direct applications of Theorem \ref{thm:lambda_{P_P}} to SNFs. 
For (i), we compute
the number of $n\times m$ matrices over $\mathbb{Z}/p^s \mathbb{Z}$ whose SNF is  $D_{\bm a}$ by \cite[Theorem 1]{FNKS13} (or \cite[Theorem 2]{FNKS14}) and simplify it to 
\begin{equation}\label{eq:N}
p^{\sum_{i=1}^s \left[ (n+m)a_i - a_i^2\right] } \cdot \frac{\prod_{i=0}^{a_s-1} (1-p^{-n+j}) (1-p^{-m+j})}{\prod_{i=0}^{s-1}\prod_{j=1}^{a_{i+1}-a_i}(1-p^{-j}) }=: N .
\end{equation}
Thus 
\begin{equation*}
\mu_{p^s}(\{D_{\bm a}\})=p^{-snm}N=\mathrm{RHS\ of}\ \eqref{eq:mu_{p^s}(D)}\,.
\end{equation*} 
\end{proof}

\begin{remark}  
In the case of $s=1$, the formula \eqref{eq:N} gives the number of $n\times m$ matrices  over $\mathbb{Z}/p \mathbb{Z}$  of rank $a_1$ and is consistent with  \cite[Exercise 1.192(b)]{S}; whereas in the case of $n=m$, a  calculation shows that \eqref{eq:N} is consistent with the results in \cite[pp.\,233, 236]{FW} (their $|\mathrm{Aut}\, H|$ is our $N$). 
\end{remark}

\subsection{\texorpdfstring{SNF Distribution over $\mathbb{Z}$\,}{SNF Distribution over Z}}
$ $

The properties of $\lambda$ of set unions, subtractions and complements in Section \ref{sec:properties of lambda} also carry over to SNFs.
They will be useful in determining the value of $\mu$ for some specific sets (for instance, the singleton set of the identity matrix as in Section \ref{sec:1}).

\begin{theorem}\label{thm:mu(union)}
Suppose that $\{\mathcal{S}_{\alpha}\}_{\alpha\in \mathcal{A}}$ are pairwise disjoint subsets of\, $\mathbb{S}$ such that $\mu(\mathcal{S}_\alpha)$ exists for all $\alpha\in \mathcal{A}$\,. If $\mathcal{A}$ is a  finite set, then
\begin{equation*}
\mu\left( \cup_{\alpha\in \mathcal{A}}\, \mathcal{S}_{\alpha}\right) =\sum_{\alpha\in \mathcal{A}} \mu(\mathcal{S}_{\alpha})\,.
\end{equation*}
\end{theorem}

\begin{theorem}\label{thm:mu(subtraction)}
Suppose that $\mathcal{S}'\subseteq \mathcal{S}\subseteq \mathbb{S}$ such that $\mu(\mathcal{S}')$ and $\mu(\mathcal{S})$ both exist, then
\begin{equation*}
\mu( \mathcal{S}\setminus \mathcal{S}') =\mu(\mathcal{S})-\mu(\mathcal{S}')\,.
\end{equation*}
In particular for the complement $\mathcal{S}^c$ of $\mathcal{S}$ in $\mathbb{S}$, we have
\begin{equation*}
\mu( \mathcal{S}^c) =1-\mu(\mathcal{S})\,.
\end{equation*}
\end{theorem}

\begin{theorem}\label{thm:subset}
Suppose that $\mathcal{T}\in \mathbb{S}$ such that $\mu(\mathcal{T})=0$, then for any $\mathcal{S}\subseteq \mathcal{T}$, we also have $\mu(\mathcal{S})=0$\,.
\end{theorem}

\subsection{\texorpdfstring{Connection between $\mu$ and $\mu_{p^s}$}{Connection between mu and mu{\_}\{p{\^{}}s\}}}
$ $

Theorem \ref{thm:lambda(Z)} has an analogue for SNFs as well, by virtue of the following well-known lemma (see \cite[Theorem 61.1]{B} for an easy proof).

\begin{lemma}\label{lemma:irred}
Fix a positive integer $r$. The determinant of an $r\times r$ matrix as a polynomial of its $r^2$ entries $x_1,x_2,\dots,x_{r^2}$ is irreducible in $\mathbb{Q}[x_1,x_2,\dots,x_{r^2}]$\,.
\end{lemma}

For any $i\le m\wedge(n-1)$ (i.e., $\min\,  \{m,n-1\}$, which is $m$ if $m<n$, and  $n-1$ if $m=n$, recalling that $m\le n$), the set $U_i$ contains at least two different minors, which are both irreducible as polynomials of the entries on the strength of Lemma \ref{lemma:irred} and therefore relatively prime. Hence Assumption \ref{asmp:coprime} holds with $w=m\wedge(n-1)$. This allows us to apply Theorem \ref{thm:lambda(Z)} to SNFs and obtain the following analogue. In addition, we will compute the density $\mu(\mathcal{S})$ explicitly later in Section \ref{sec:mu(S)}.

\begin{theorem}\label{thm:mu(S)}
Given positive integers $r\le m\wedge(n-1)$ and $d_1\, | d_2 \,| \cdots | \,d_r$, let $z=\prod_{j=1}^{\infty} p_j^{s_j}$ with $p_j$ the $j$-th smallest prime and $s_j$ nonnegative integers, $j=1,2,\dots$ such that $d_r\,|\,z$,
then the probability $\mu(\mathcal{S})$ exists for 
\begin{equation}\label{eq:S}
\mathcal{S}=\left\{D:=(D_{i,j})_{n\times m} \in \mathbb{S}:\ D_{i,i}=d_i\,,\  \forall\ i\le r\right\},
\end{equation} 
and in fact 
\begin{equation}\label{eq:mu(S)}
\mu(\mathcal{S}) = \prod_{j=1}^{\infty} \mu_{p_j^{s_{j}+1}}(\mathcal{S})\,.
\end{equation}
\end{theorem}

\begin{remark}\label{rmk:mu0}
{\rm
(i) The right-hand side of \eqref{eq:mu(S)} is well-defined since $\mu_{p^s}(\cdot)\in [0,1]$ for all $p$ and $s$.

(ii) We have assumed that $r\le m\wedge(n-1)$; 
in fact, we have $\mu(\mathcal{S})=0$ otherwise. Recall that $m\le n$ and note that $r\le m$, thus in the case of $r> m\wedge(n-1)$, we must have $r=m=n$. 
As a result, any matrix $M$ with $\mathrm{SNF}(M)\in \mathcal{S}$ satisfies $|M|=\pm d_n$\,. 
We will show later that  the probability that the determinant of a random $n\times n$ integer matrix equals $\pm c$ is $0$ for all constant $c$ (Theorem \ref{thm:det}). 

(iii) We have also assumed that the $d_i$'s are positive; in fact, we have $\mu(\mathcal{S})=0$  otherwise. If $d_i=0$ for some $i$, then all $i\times i$ minors of any matrix $M$ with $\mathrm{SNF}(M)\in \mathcal{S}$ are zero. Applying Theorem \ref{thm:det} to $c=0$ yields the desired result. }
\end{remark}

\section{Applications}\label{sec:app}
Now we apply Theorems \ref{thm:p^s} and \ref{thm:mu(S)}  to compute the density $\mu$ explicitly for the following subsets of $\mathbb{S}$: 
matrices with first few diagonal entries given (i.e., with the form of \eqref{eq:S}),
full rank matrices,
a finite subset, 
matrices with diagonal entries all equal to $1$,
and square matrices with at most  $\ell\,(=1,2,\dots,n)$ diagonal entries not equal to $1$.

\subsection{\texorpdfstring{Density of the Set \eqref{eq:S}}{Density of the Set (3.4)}}\label{sec:mu(S)}
$ $

For the set $\mathcal{S}$ of \eqref{eq:S}, i.e., of matrices with first $r$ diagonal entries given, 
we take $z=d_r$ in Theorem \ref{thm:mu(S)}, then it suffices to compute $\mu_{p^{s+1}}(\mathcal{S})$ for each $(p,s)=(p_j,s_j)$. 
In  ${\rm{mod}}\ p^{s+1}$, the set $\mathcal{S}$ has $m-r+1$ elements (see \eqref{eq:S=} below). 
Further, since formula \eqref{eq:mu_{p^s}(D)} gives the density $\mu_{p^{s+1}}$ of each element of $\mathcal{S}$, 
one can take the sum over  $\mathcal{S}$ to get an expression for $\mu_{p^{s+1}}(\mathcal{S})$ (Theorem \ref{thm:p^s}), and compute this sum explicitly when $m-r$ is small, such as in Theorems \ref{thm:1} and \ref{thm:cyclic SNF} below.
However, this sum is hard to compute when $m-r$ is large, for example, when $m$ is large and $r$ is fixed; 
in this case, we recast $\mathcal{S}$ as the difference between a subset of $\mathbb{S}$ and the union of other $r-1$ subsets such that for each of these $r$ sets, its density $\mu_{p^{s+1}}$  is given directly by \eqref{eq:mu_{p^s}(D)}. 

We work out two examples to illustrate this idea,  and then deal with the general case.

\subsubsection{The First Example: Relatively Prime Entries}
$ $

Our approach reproduces the following result mentioned at the beginning of this paper.
\begin{theorem}
Let $\mathcal{S}$ be the set of \eqref{eq:S} with $r=1$ and $d_1=1$, then we have\begin{equation}\label{eq:mu(1)}
\mu(\mathcal{S})
=\frac{1}{\zeta(nm)}\,,
\end{equation}
where $\zeta(\cdot)$ is the Riemann zeta function.
\end{theorem}
\begin{proof}
Applying Theorem \ref{thm:mu(S)} with $r=1$\,, $d_1=1$ and $s_j=0$\,, $j=1,2,\dots$ gives
\begin{equation} \label{eq:prod mu_p}
\mu(\mathcal{S})
=\prod_{p}\mu_{p}(\mathcal{S})\,,
\end{equation}
therefore it reduces to computing $\mu_{p}(\mathcal{S})$ for each $p$.

Recall the equivalence of multiplication by units, therefore we only have two choices for matrix entries  in ${\rm{mod}}\ p$\,: $1$ and $0$.
The set $\mathcal{S}$ $({\rm{mod}}\ p)$  consists of all the matrices in $\mathbb{S}$ whose first diagonal entry is $1$, 
thus $\mathcal{S}=\{D_{\bm{a}}: \bm{a}=(a_1)\ge 1\}$ $({\rm{mod}}\ p)$ 
(recall from Theorem \ref{thm:p^s} that ${\bm a}:=(a_1,a_2,\dots,a_s)$ and that $D_{\bm a}\in \mathbb{S}$ is the diagonal matrix with exactly $a_i$ non-$p^i$-multiples on its diagonal).
Therefore
\begin{equation*} 
\mu_{p}(\mathcal{S})
=1-\mu_{p}(\{D_{(0)}\})\,.
\end{equation*}
We apply \eqref{eq:mu_{p^s}(D)} to get $\mu_{p}(\{D_{(0)}\})=p^{-nm}$, thus $\mu_{p}(\mathcal{S})=1-p^{-nm}$. 
Plugging into \eqref{eq:prod mu_p} along with the Euler product formula
\begin{equation}\label{eq:Euler}
\prod_{p} \left( 1-p^{-i}\right) =\frac{1}{\zeta(i)}\in(0,1)\,,\ \ \forall \ i\ge 2
\end{equation} 
yields \eqref{eq:mu(1)}.
\end{proof}

\subsubsection{Another Example}

\begin{theorem}\label{thm:mu(2,6)}
Let $\mathcal{S}$ be the set of \eqref{eq:S} with $r=2$, $d_1=2$ and $d_2=6$, then we have
\begin{equation}\label{eq:mu(2,6)}
\mu(\mathcal{S})
=\mu_{2^2}(\mathcal{S})\,\mu_{3^2}(\mathcal{S})\prod_{p>3}\mu_{p}(\mathcal{S})\,,
\end{equation}
where
\begin{equation}\label{eq:mu_4}
\mu_{2^2}(\mathcal{S})
=2^{-nm}  \left( 1 - 2^{-nm} - 2^{-(n-1)(m-1)} \cdot \frac{\left( 1-2^{-n}\right) \left( 1-2^{-m}\right)}  {1-2^{-1}}\right) ,
\end{equation}
\begin{equation}\label{eq:mu_9}
\mu_{3^2}(\mathcal{S})
=3^{-(n-1)(m-1)}  \left( 1 - 3^{-(n-1)(m-1)}\right) \frac{\left( 1-3^{-n}\right) \left( 1-3^{-m}\right)} {1-3^{-1}}\,,
\end{equation}
\begin{equation}\label{eq:mu_p}
\mu_{p}(\mathcal{S})
=1-p^{-nm} - p^{-(n-1)(m-1)}\cdot \frac {\left( 1-p^{-n}\right) \left( 1-p^{-m}\right)}{1-p^{-1}} =1 -\sum_{(n-1)(m-1)}^{(n-1)m}p^{-i}  +
\sum_{n(m-1)+1}^{nm-1} p^{-i}\,.
\end{equation}
\end{theorem}

\begin{proof}
The first equation \eqref{eq:mu(2,6)} follows directly from Theorem \ref{thm:mu(S)} with $r=2$\,, $d_1=2$\,, $d_2=z=6$\,, $s_1=s_2=1$\,, $s_j=0$\,, $j\ge 3$\,.
Therefore it reduces to calculating $\mu_{p^s}(\mathcal{S})$ for $(p,s)=(2,2)$, $(3,2)$ and $(p,1)$ with $p>3$\,.

\smallskip
\noindent
{\it Case 1}.
$p>3$ and $s=1$\,.

Recall the equivalence of multiplication by units, therefore we only have two choices for matrix entries  in ${\rm{mod}}\ p$\,: $1$ and $0$.
The set $\mathcal{S}$ $({\rm{mod}}\ p)$  consists of all the matrices in $\mathbb{S}$ whose first two diagonal entries are  $1$, 
thus $\mathcal{S}=\{D_{\bm{a}}: \bm{a}=(a_1)\ge 2\}$ $({\rm{mod}}\ p)$ 
(recall ${\bm a}$ again from Theorem \ref{thm:p^s}).
Therefore
\begin{equation*} 
\mu_{p}(\mathcal{S})
=1-\mu_{p}(\{D_{(0)}\}) -\mu_{p}(\{D_{(1)}\})\,.
\end{equation*}
We then apply \eqref{eq:mu_{p^s}(D)} to get $\mu_{p}(\{D_{(0)}\})$ and $\mu_{p}(\{D_{(1)}\})$, and  \eqref{eq:mu_p} follows.

\smallskip
\noindent
{\it Case 2}.
$p=2$ and $s=2$\,.

We have three choices for matrix entries  in ${\rm{mod}}\ 2^2$: $1$, $2$ and $0$. The set $\mathcal{S}$ $({\rm{mod}}\ 2^2)$  consists of all the  matrices in $\mathbb{S}$ whose first two diagonal entries are $2$, thus $\mathcal{S}=\{D_{\bm{a}=(a_1,a_2)}: a_1=0\,,\, a_2\ge 2\}$ $({\rm{mod}}\ 2^2)$. Therefore
\begin{equation}\label{eq:mu_4=}
\mu_{2^2}(\mathcal{S})
=\mu_{2^2}(\{D_{(a_1,a_2)}: a_1=0\}) -\mu_{2^2}(\{D_{(0,0)}\})
-\mu_{2^2}(\{D_{(0,1)}\})\,.
\end{equation}
Notice that the set $\{D_{(a_1,a_2)}: a_1=0\}$ $({\rm{mod}}\ 2^2)$ 
consists of all the matrices in $\mathbb{S}$ whose diagonal entries are all multiples of $2$ (i.e., either $2$ or $0$); 
in other words, in ${\rm{mod}}\ 2$, it contains only one element -- the zero matrix. Hence
\begin{equation*}
\mu_{2^2}(\{D_{(a_1,a_2)}: a_1=0\})
=\mu_{2}(\{D_{(0)}\})\,.
\end{equation*}
Plugging into \eqref{eq:mu_4=} and applying \eqref{eq:mu_{p^s}(D)} to get  $\mu_{2}(\{D_{(0)}\})$, $\mu_{2^2}(\{D_{(0,0)}\})$ and $\mu_{2^2}(\{D_{(0,1)}\})$, we obtain \eqref{eq:mu_4}.

\smallskip
\noindent
{\it Case 3}.
$p=3$ and $s=2$\,.

We have three choices for matrix entries  in ${\rm{mod}}\ 3^2$: $1$, $3$ and $0$.
The set $\mathcal{S}$ $({\rm{mod}}\ 3^2)$  consists of all the  matrices in $\mathbb{S}$ whose first two diagonal entries are $1$ and $3$, respectively, thus $\mathcal{S}=\{D_{\bm{a}=(a_1,a_2)}: a_1=1\,,\, a_2\ge 2\}$ $({\rm{mod}}\ 3^2)$.
Therefore
\begin{equation}\label{eq:mu_9=}
\mu_{3^2}(\mathcal{S})
=\mu_{3^2}(\{D_{(a_1,a_2)}: a_1=1\}) -\mu_{3^2}(\{D_{(1,1)}\})\,.
\end{equation}
Notice that the set $\{D_{(a_1,a_2)}: a_1=1\}$ $({\rm{mod}}\ 3^2)$ 
consists of all the matrices in $\mathbb{S}$ whose first diagonal entry is $1$ and all other diagonal entries are multiples of $3$ (i.e., either $3$ or $0$); 
in other words, in ${\rm{mod}}\ 3$, it contains only one element -- the diagonal matrix whose diagonal entries are $1,0,0,\dots,0$. Hence
\begin{equation*}
\mu_{3^2}(\{D_{(a_1,a_2)}: a_1=1\})
=\mu_{3}(\{D_{(1)}\})\,.
\end{equation*}
Plugging into \eqref{eq:mu_9=} and applying \eqref{eq:mu_{p^s}(D)} to get $\mu_{3}(\{D_{(1)}\})$ and $\mu_{3^2}(\{D_{(1,1)}\})$, we obtain \eqref{eq:mu_9}.
\end{proof}

\subsubsection{The General Case}

\begin{theorem}\label{thm:mu}
Let $\mathcal{S}$ be the set of \eqref{eq:S} in Theorem \ref{thm:mu(S)}  with $d_r=\prod_{j=1}^{\infty}p_j^{s_j}$, then for $(p,s)=(p_j,s_j)$, $j=1,2,\dots$, we have
\begin{align}\label{eq:mu_{p^s}}
\nonumber \mu_{p^{s+1}}(\mathcal{S})
=&\ p^{- \sum_{i=1}^{s}(n-{\tilde{a}_i})(m-{\tilde{a}_i})} \cdot 
\frac{[p,n][p,m]}
{[p,n-{\tilde{a}_s}][p,m-\tilde{a}_s] \prod_{i=1}^{s} [p,{\tilde{a}_i}-{\tilde{a}_{i-1}}]}\\
&\ -\sum_{\ell=\tilde{a}_s}^{r-1} p^{-(n-\ell)(m-\ell)- \sum_{i=1}^{s}(n-{\tilde{a}_i})(m-{\tilde{a}_i})} \cdot 
\frac{[p,n][p,m]}
{[p,n-\ell][p,m-\ell][p,\ell-\tilde{a}_s] \prod_{i=1}^{s} [p,{\tilde{a}_i}-{\tilde{a}_{i-1}}]}\,,
\end{align}
where ${\tilde{a}_i}$ $(0\le i\le s)$ is the number of non-$p^i$-multiples among $d_1,d_2,\dots,d_r$ (thus $\tilde{a}_s\le r-1$)\,. 
In particular, when $s=0$ (which holds for all but finitely many $j$'s), we have
\begin{equation}\label{eq:mu_{p}}
\mu_{p}(\mathcal{S})
=1 -\sum_{\ell=0}^{r-1} p^{-(n-\ell)(m-\ell)} \cdot 
\frac{[p,n][p,m]}
{[p,n-\ell][p,m-\ell][p,\ell]}\,.
\end{equation}
The value of  $\mu(\mathcal{S})$ is then given by Theorem \ref{thm:mu(S)} with $z=d_r$\,.
\end{theorem}
\begin{proof} 
Recalling from Theorem \ref{thm:p^s} the notation of $D_{\bm a}$, we recast $\mathcal{S}$ as
\begin{equation}\label{eq:S=}
\mathcal{S}=\{D_{\bm{a}=(a_1,a_2,\dots,a_{s+1})}: a_i={\tilde{a}_i}\,,\, 1\le i\le s\,,\, a_{s+1}\ge r\}\ ({\rm{mod}}\ p^{s+1})\,,
\end{equation} 
and therefore
\begin{equation}\label{eq:mu_{p^s}=}
\mu_{p^{s+1}}(\mathcal{S})
=\mu_{p^{s+1}}(\{D_{\bm{a}=(a_1,a_2,\dots,a_{s+1})}: a_i={\tilde{a}_i}\,,\, 1\le i\le s\}) -\sum_{\ell=\tilde{a}_s}^{r-1} \mu_{p^{s+1}}(\{D_{({\tilde{a}_1},{\tilde{a}_2},\dots,{\tilde{a}_s},\ell)}\})\,.
\end{equation}
Notice that the set $\{D_{\bm{a}=(a_1,a_2,\dots,a_{s+1})}: a_i={\tilde{a}_i}\,,\, 1\le i\le s\}$ $({\rm{mod}}\ p^{s+1})$  in the first term on the right-hand side of \eqref{eq:mu_{p^s}=} consists of all the matrices in $\mathbb{S}$ with exactly ${\tilde{a}_i}$ $(1\le i\le s)$ non-$p^i$-multiples on its diagonal; 
in other words, in ${\rm{mod}}\ p^s$, it contains only one element -- the diagonal matrix with exactly ${\tilde{a}_i}$ non-$p^i$-multiples, i.e.,  $({\tilde{a}_i}-{\tilde{a}_{i-1}})$ $p^{i-1}$'s,  $1\le i\le s$ on its diagonal. Hence
\begin{equation}
\mu_{p^{s+1}}(\{D_{\bm{a}=(a_1,a_2,\dots,a_{s+1})}: a_i={\tilde{a}_i}\,,\, 1\le i\le s\})
=\mu_{p^s}(\{D_{({\tilde{a}_1},{\tilde{a}_2},\dots,{\tilde{a}_s})}\})\,.
\end{equation}
Plugging into \eqref{eq:mu_{p^s}=} and applying \eqref{eq:mu_{p^s}(D)} to get  $\mu_{p^s}(\{D_{({\tilde{a}_1},{\tilde{a}_2},\dots,{\tilde{a}_s})}\})$ and $\mu_{p^{s+1}}(\{D_{({\tilde{a}_1},{\tilde{a}_2},\dots,{\tilde{a}_s},\ell)}\})$, $1\le \ell\le r-1$, we obtain \eqref{eq:mu_{p^s}}.
\end{proof}

\begin{remark} 
We notice that the density $\mu_{p^s}(\{D_{\bm a}\})$ of \eqref{eq:mu_{p^s}(D)} is a polynomial of $p^{-1}$ with integer coefficients since $m-a_s+\sum_{i=1}^s (a_i-a_{i-1})=m$.
The $\mu_{p}(\mathcal{S})$ of 
\eqref{eq:mu_{p}} is also a polynomial of $p^{-1}$ with integer coefficients and with constant term $1$ (see the $\mu_{p}(\mathcal{S})$ of \eqref{eq:mu_p} as an example).
If we replace each occurrence of $p$ by $p^{\boldsymbol{z}}$, where $\boldsymbol{z}$ is a complex variable, and plug into \eqref{eq:mu(S)} of Theorem \ref{thm:mu(S)}, 
we get an Euler product for some kind of generalized zeta function. 

For instance, when $m=n=3$, for the set $\mathcal{S}$ in Theorem \ref{thm:mu(2,6)}, we apply \eqref{eq:mu_p} to get
\begin{equation*}
\mu_{p}(\mathcal{S})
=1-p^{-4}-p^{-5}-p^{-6}+p^{-7}+p^{-8}
=\left( 1-p^{-2}\right) \left(1-p^{-3} \right) \left(1+p^{-2}+ p^{-3} \right).
\end{equation*}
Taking the product over all primes $p$ and applying the Euler product formula \eqref{eq:Euler}, we obtain
\begin{equation*}
\prod_p\mu_{p}(\mathcal{S})
=\frac{1}{\zeta(2)\zeta(3)} \prod_p\left(1+p^{-2}+ p^{-3} \right) .
\end{equation*}
Plugging into \eqref{eq:mu(2,6)}, we see that to obtain the density $\mu(\mathcal{S})$, it reduces to computing $\prod_p (1+p^{-2}+p^{-3})$, or to understanding the Euler product   $\prod_p(1+p^{-2\boldsymbol{z}}+p^{-3\boldsymbol{z}})$. 

It would be interesting to study whether such an Euler product for some generalized zeta function 
(1) has any interesting properties relevant to SNF; 
(2) extends to a meromorphic function on all of $\mathbb{C}$\,;
(3) satisfies a functional equation.
\end{remark}

\subsection{The Determinant}\label{sec:det}
$ $

The determinant of an $m\times m$ matrix can be regarded as a polynomial $G$ of its $m^2$ entries. 
Note that $G$ is not a constant since it takes  values $1$ and $0$ for the identity matrix and the zero matrix, respectively. Thus we can 
apply Theorem \ref{thm:G=0} to $G$ and obtain the following.

\begin{theorem}\label{thm:det}
Let $c$ be an integer. The probability that the determinant  equals $c$ for an $m\times m$ matrix with entries from $\mathbb{Z}_{(k)}$ goes to $0$ as $k\to \infty$\,; in other words, the density of the determinant of a random $m\times m$ integer matrix is always $0$. 
\end{theorem}

This result plays an important role in the next two theorems. The first of them shows that
the probability that a random $n\times m$ integer matrix is full rank is $1$.

\begin{theorem}\label{thm:D_{m,m}=0}
If \,$\mathcal{S}\subseteq \mathbb{S}$ satisfies  $D_{m,m}=0$ for all $D=(D_{i,j})_{n\times m}\in \mathcal{S}$, then we have
$\mu(\mathcal{S}) =0$\,; 
in other words, the probability that an  $n\times m$ matrix with entries from $\mathbb{Z}_{(k)}$ is  full rank goes to $1$ as $k\to \infty$\,.
\end{theorem}

\begin{proof}
If $\mathrm{SNF}(M)_{m,m}=0$\,, then all $m\times m$ minors of $M$ are zero.  
Therefore the result  follows from Theorem \ref{thm:det} with $c=0$\,.
\end{proof}

When $m=n$, we can generalize Theorem \ref{thm:D_{m,m}=0} to $\mathcal{S}$ with  finitely many values of $D_{m,m}$'s.

\begin{theorem}\label{thm:mu(D)}
Suppose that $m=n$ and \,$\mathcal{S}\subset \mathbb{S}$\,, then we have 
$\mu(\mathcal{S}) =0$ if the set $\{D_{n,n}:D=(D_{i,j})_{n\times n}\in \mathcal{S}\}$ is finite; 
in particular, this holds for any finite subset $\mathcal{S}\subset \mathbb{S}$\,.
\end{theorem}

\begin{proof} 
For any $M$ such that $\mathrm{SNF}(M)=D\in \mathcal{S}$, we have 
$|M|=\pm D_{1,1}D_{2,2}\cdots D_{n,n}$\,.
As a consequence, if $D_{n,n}=0$\,, then $|M|=0$\,; 
if $D_{n,n}\neq 0$\,, then the $D_{i,i}$'s are divisors of $D_{n,n}$ and therefore $|M|$ has finitely many choices. 
The result then follows from Theorem \ref{thm:det}.
\end{proof}

If $D_{n,n} \neq 0$ for all $D\in \mathcal{S}$, then we have another proof of Theorem \ref{thm:mu(D)} without invoking Theorem \ref{thm:det}.
We cannot take advantage of \eqref{eq:mu_{P_P}} from Theorem \ref{thm:mu(S)} since $r=m=n>m\wedge(n-1)$ in this case. Instead, we will start from the observation that $\mu^{(k)}(\mathcal{S}) \le \mu^{(k)}_{P(\ell)}(\{I\})$ with $P(\ell)$ a product of primes and $I$ the identity matrix, then bound $\mu^{(k)}_{P(\ell)}(\{I\})$ from above by  $2^{n^2} \mu_{P(\ell)}(\{I\})$ through a similar idea as in the proof of \eqref{eq:lambda_{P_P}(Z)} (approximating  $2k+1$ by a multiple of $P(\ell)$),  and finally show that $\mu_{P(\ell)}(\{I\})\to 0$ as $\ell\to \infty$.

\begin{proof}[Another Proof of Theorem \ref{thm:mu(D)} with $D_{n,n} \neq 0$  for all $D\in \mathcal{S}$\,]

Let $I$ be the $n\times n$ identity matrix and $j^*\in \mathbb{Z}_+$ such that $p_{j}> c$ for all $j\ge j^*$. Then for any $j>j^*$, SNF$(M)\in \mathcal{S}$ (mod $p_j$) only if SNF$(M)=I$ (mod $p_j$).  

Applying \eqref{eq:N} with $s=1$ and $a_1=n$ (or \cite[Exercise 1.192(b)]{S}), we get the number of  $n\times n$ non-singular matrices over $\mathbb{Z}/p_j\mathbb{Z}$\,:
\begin{equation*}
p_j^{n^2}[p_j,n]:=\beta_j\,.
\end{equation*}  

Set 
\begin{equation*}
P(\ell) := p_{j^*}\, p_{j^*+1} \cdots p_{\ell}\,,\ \ell\ge j^*.
\end{equation*} 
Then SNF$(M)\in \mathcal{S}$ (mod $P(\ell)$) only if SNF$(M)=I$ (mod $P(\ell)$).  Hence $\mu_{P(\ell)}^{(k)}(\mathcal{S})\le \mu_{P(\ell)}^{(k)}(\{I\})$.

By the Chinese remainder theorem, the number of $n\times n$ matrices over $\mathbb{Z}/P(\ell)\mathbb{Z}$ whose SNF  equals $I$ mod $P(\ell)$ is  \begin{equation}\label{eq:beta(l)}
\beta_{j^*} \beta_{j^*+1}\cdots \beta_{\ell}=\prod_{j=j^*}^{ \ell}\, p_j^{n^2}[p_j,n]=P(\ell)^{n^2}\prod_{j=j^*}^{ \ell}[p_j,n]:=\beta(\ell)\,.
\end{equation}  

For any integer $k$ with $K:=2k+1 >P(\ell)$,  there exists $q\in \mathbb{Z}_+$ such that
\begin{equation}\label{eq:q''}
q\cdot P(\ell) \le  K < (q+1)\cdot P(\ell)\,. 
\end{equation}   
Then for any integer $z$\,, there are at most $q+1$ numbers among $\mathbb{Z}_{(k)}$ that equal  $z$ mod $P(\ell)$. 
Therefore the number of $n\times n$ matrices with entries from $\mathbb{Z}_{(k)}$ whose SNF  is equal to $I$ mod $P(\ell)$   is at most $(q+1)^{n^2} \beta(\ell)$. Hence
\begin{equation}\label{eq:mu_k(I)}
\mu_{P(\ell)}^{(k)}(\{I\})\le \frac{(q+1)^{n^2} \beta(\ell)}{ K^{n^2} }\le \frac{(q+1)^{n^2} \beta(\ell)}{ \left[ qP(\ell)\right] ^{n^2} }
=\left(\frac{q+1}{q}\right) ^{n^2} \frac{\beta(\ell)}{P(\ell)^{n^2} }
\le 2 ^{n^2}\prod_{j=j^*}^{ \ell} [p_j,n]\,,
\end{equation}
on the strength of \eqref{eq:q''} and \eqref{eq:beta(l)} (note that $P(\ell)^{-n^2}\beta(\ell)=\mu_{P(\ell)}(\{I\})$ by \eqref{eq:beta(l)} and \eqref{eq:mu_{P_P}}).

Notice that 
\begin{equation}\label{eq:1-x}
1- x\le \exp(-x)\,,\quad \forall\ x\in[0,1]\,.
\end{equation}
(To see this, let $W(x):=1- x-\exp(-x)$, $x\in[0,1]$, then $W'(x)=-1+\exp(-x)\le 0$\,. Hence $W(x)\le W(0)=0$\,.)

Applying \eqref{eq:1-x} with $x=p_j^{-1}$ $(j^*\le  j\le \ell)$, we obtain  
$$[p_j,n]\le 1- p_j^{-1}\le  \exp\left( -p_j^{-1}\right).$$
Plugging into \eqref{eq:mu_k(I)} yields
\begin{equation*}
\mu_{P(\ell)}^{(k)}(\{I\})
\le 2^{n^2} \prod_{1\le  j\le \ell}\ \exp\left( -p_j^{-1}\right) 
=2^{n^2} \exp\left( -\sum_{j^*\le  j\le \ell} p_j^{-1}\right)
\to 0\ \ \  \text{as} \ \ \ell\to \infty  
\end{equation*}
with $2k+1\,(=K)>P(\ell)$, 
by the well-known result that 
\begin{equation*}
\sum_{1\le  j\le \ell} p_j^{-1}\to\infty \ \ \ \text{as} \ \ \ell\to \infty\,.
\end{equation*}
Since $\mu^{(k)}(\mathcal{S})\le \mu_{P(\ell)}^{(k)}(\mathcal{S})\le \mu_{P(\ell)}^{(k)}(\{I\}) $, we deduce that 
$\mu^{(k)}(\mathcal{S})\to 0$ as $k \to \infty$\,, as desired.
\end{proof}

\subsection{\texorpdfstring{Probability that All Diagonal Entries of the SNF Are $1$}{Probability that All Diagonal Entries of the SNF Are 1}} \label{sec:1}
$ $

Theorem \ref{thm:mu(D)} (along with Theorem \ref{thm:subset}) implies that the probability that all diagonal entries of an SNF are $1$ is $0$ if $m=n$\,; 
however, as we will see soon,  this probability is positive if $m<n$\,. We will need Theorems   \ref{thm:p^s} and \ref{thm:mu(S)} to determine its value.

\begin{theorem}\label{thm:1}
Let $E$ be the $n\times m$ diagonal matrix whose diagonal entries are all $1$. 
If $m<n$\,, then we have
\begin{equation*}
\mu(\{E\}) = \frac{1}{\prod_{i=n-m+1}^n \zeta(i)} \to
\begin{cases} 
1\,, & {\mathrm{if}}\ m\  {\mathrm{is\ fixed}}\ \\ 
\frac{1}{\prod_{i=n-m+1}^{\infty} \zeta(i)}\,, & {\mathrm{if}}\ n-m\  {\mathrm{is\ fixed}}\ 
\end{cases}\,,\ \ {\mathrm{as}}\ \ n\to \infty\,.
\end{equation*} 
\end{theorem}

\begin{proof}
Apply Theorem \ref{thm:mu(S)} with $\mathcal{S}=\{E\}$, $r=m$\,, $d_i=z=1$\,, $s_j=0$ for all $i,j$, and then Theorem \ref{thm:p^s} with $s=1$\,, $a_1=m$\,:
\begin{align*}
\mu(\{E\}) =&\ \prod_{p} \mu_{p}(\{E\})
=\prod_{p} \frac{[p,n]}{[p,n-m]}
= \prod_{p}\prod_{i=n-m+1}^n \left( 1-p^{-i}\right) =\prod_{i=n-m+1}^n  \prod_{p} \left( 1-p^{-i}\right)\\
=&\ \frac{1}{\prod_{i=n-m+1}^n \zeta(i)}\,,
\end{align*}
on the strength of $n-m+1\ge 2$ and the Euler product formula \eqref{eq:Euler}. 

Finally, thanks to the fact that $\zeta(i)\downarrow 1$ as $i\to \infty$\,, we obtain the limits of $\mu(\{E\})$ as desired.
\end{proof}

\subsection{\texorpdfstring{Probability that At Most $\ell$ Diagonal Entries of the SNF Are Not $1$}{Probability that At Most l Diagonal Entries of the SNF Are Not 1}}\label{sec:l}
$ $

In this section, we assume that $m=n$.  We provide a formula for the
probability that an SNF has at most $\ell$ diagonal entries not equal
to $1$ and a formula for the limit of this probability as $n\to
\infty$\,.  In particular, when $\ell=1$, this limit is the reciprocal
of a product of values of the Riemann zeta function at positive
integers and equals $0.846936$.  For bigger $\ell$, we prove that this
limit converges to $1$ as $\ell\to \infty$ and find its asymptotics
(see \eqref{eq:asymZ(l)}).

\subsubsection{\texorpdfstring{Cyclic SNFs ($\ell=1$)}{Cyclic SNFs (l=1)}}\label{sec:cyclic}
$ $

We shall say that an SNF is {\it cyclic} if it has at most one
diagonal entry not equal to $1$, i.e., if the corresponding cokernel
is cyclic.  Denote the set of $n\times n$ cyclic SNFs by
$\mathcal{T}_n$\,.  We will compute the probability $\mu
(\mathcal{T}_n)$ of having a cyclic SNF, and show that this
probability strictly decreases to $0.846936\cdots$ as $n\to \infty$\,.
As mentioned above, this result was first obtained by Ekedahl
\cite[Section 3]{E}. Later Nguyen and Shparlinski \cite[(1.2)]{NS}
showed that if take a subgroup of $\mathbb{Z}^n$ 
uniformly among all
subgroups of index at most $V$ and let $V\to\infty$, then the
probability that the quotient group is cyclic is also
$\mu(\mathcal{T}_n)$.  This result is equivalent to computing the
probability that an $n\times n$ integer matrix has a cyclic cokernel
using a certain probability distribution different from $\mu$.  We do
not know a simple reason why these two probability distributions yield
the same probability of a cyclic cokernel. Perhaps there is a
universality result which gives the same conclusion for a wide class
of probability distributions.

\begin{theorem}\label{thm:cyclic SNF}
We have

\noindent
\rm{(i)}
\begin{equation}\label{eq:Z_n}
\mu (\mathcal{T}_n) = \frac{1}{\prod_{i=2}^{n}\zeta(i)} \cdot \prod_{p} \left(1+\frac{1}{p^2} + \frac{1}{p^3} + \cdots + \frac{1}{p^n}\right)=:Z_n\,;
\end{equation} 

\noindent
\rm{(ii)} $Z_n$ is strictly decreasing in $n$\,;

\noindent
\rm{(iii)} $$Z_2=\frac{1}{\zeta(4)}=\frac{90}{\pi^4}\approx0.923938\,;$$ 

\noindent
\rm{(iv)}
\begin{equation*}
\lim_{n\to \infty}Z_n=\frac{1}{\zeta(6)\prod_{i=4}^{\infty}\zeta(i)}\approx0.846936\,.
\end{equation*}
\end{theorem} 

\begin{proof} (i)
Apply Theorem \ref{thm:mu(S)} with $\mathcal{S}=\mathcal{T}_n$\,, $r=n-1$\,, $d_i=z=1$\,, $s_j=0$ for all $i,j$, and then Theorem \ref{thm:p^s} with  $s=1$\,,  $a_1=n$\,, $n-1$\,, respectively:
\begin{eqnarray}\label{eq:Z_nPf}
\ \ && \mu(\mathcal{T}_n) \,=\, \prod_{p} \mu_{p}(\mathcal{T}_n)
\,=\,\prod_{p}\left([p,n]+ \frac{p^{-1}[p,n]^2}{[p,1]^2[p,n-1]}\right)
=\,\prod_{p}\frac{[p,n]}{[p,1]} \left([p,1]+ \frac{p^{-1}[p,n]}{[p,1][p,n-1]}\right)\\ \nonumber
\ \ &=\!\!& \frac{1}{\prod_{i=2}^{n}\zeta(i)}  \prod_{p} \left(1-p^{-1} + \frac{p^{-1}(1-p^{-n})}{1-p^{-1}}\right)
=\, \frac{1}{\prod_{i=2}^{n}\zeta(i)} \prod_{p}\left(1+\frac{1}{p^2} + \frac{1}{p^3} + \cdots + \frac{1}{p^n}\right)=\,Z_n\,.
\end{eqnarray}
Here in the fourth equality we used the fact that
\begin{equation}\label{eq:zeta}
\prod_{p}\frac{[p,n]}{[p,1]} 
= \prod_{p}\prod_{i=2}^n \left( 1-p^{-i}\right) 
= \prod_{i=2}^n\prod_{p} \left( 1-p^{-i}\right) =\frac{1}{\prod_{i=2}^n \zeta(i)}\,,
\end{equation}
by virtue of the Euler product formula  \eqref{eq:Euler}.

\smallskip
(ii) 
We consider the ratio: 
\begin{equation*}
\frac{Z_{n+1}}{Z_n}
=\prod_{p} \left( 1-p^{-(n+1)}\right) \cdot \frac{1+p^{-2}+p^{-3}+\cdots+p^{-(n+1)}}{1+p^{-2}+p^{-3}+\cdots+p^{-n}}\,,
\end{equation*}
thus it suffices to show 
\begin{equation}\label{eq:decrease}
\left( 1-p^{-(n+1)}\right) \cdot \frac{1+p^{-2}+p^{-3}+\cdots+p^{-(n+1)}}{1+p^{-2}+p^{-3}+\cdots+p^{-n}}<1
\end{equation}
for all $p$. 
For ease of notation, we denote $p^{-1}$ by $t$ throughout this paper, then 
\begin{equation*}
{\rm{LHS\ of}}\ \eqref{eq:decrease} 
=\left( 1-t^{n+1}\right) \cdot
\left( 1+ \frac{t^{n+1}}{1+t^{2}+t^{3}+\cdots+t^{n}}\right) 
< \left(1-t^{n+1}\right)\left(1+t^{n+1}\right)
=1-t^{2(n+1)}<1\,.
\end{equation*}

\smallskip
(iii) When $n=2$, it follows from definition \eqref{eq:Z_n} that
$$Z_2
=\prod_{p}\left( 1-p^{-2}\right) \left( 1+p^{-2}\right) 
=\prod_{p}\left( 1-p^{-4}\right) 
=\frac{1}{\zeta(4)}\,.$$

\smallskip
(iv) 
Now assume that $n\ge 3$.
According to the definition \eqref{eq:Z_n} of $Z_n$\,, it suffices to prove that
\begin{equation*}
\lim_{n\to \infty}\prod_{p} \left(1+\frac{1}{p^2} + \frac{1}{p^3} + \cdots + \frac{1}{p^n}\right)
=\frac{\zeta(2)\zeta(3)}{\zeta(6)}\,.
\end{equation*}
In fact, we will show that
\begin{equation}\label{eq:limZ}
\frac{\zeta(2)\zeta(3)}{\zeta(6)}
=\prod_{p} \left(1+\frac{1}{p^2} + \frac{1}{p^3} + \cdots \right)
=\lim_{n\to \infty}\prod_{p} \left(1+\frac{1}{p^2} + \frac{1}{p^3} + \cdots + \frac{1}{p^n}\right).
\end{equation}

We adopt the notation $t:=p^{-1}$. For the left equality of \eqref{eq:limZ}, we observe that
\begin{equation}\label{eq:Y(p,1)}
1+t^2 + t^3 + \cdots
=1+\frac{t^{2}}{1-t}
=\frac{1-t+t^{2}}{1-t}
=\frac{1+t^{3}}{(1+t)(1-t)}
=\frac{1-t^{6}}{\left( 1-t^{2}\right) \left(1-t^{3} \right)}\,.
\end{equation}
Taking the product of this equation over all  reciprocals $t$ of primes and applying
the Euler product formula \eqref{eq:Euler} yields the desire equality.

\smallskip
For the right equality of \eqref{eq:limZ},
since
\begin{equation*}
0<
1-\frac{1+t^{2}+t^{3}+\cdots+t^{n}}{1+t^{2}+t^{3}+\cdots}
=\frac{t^{n+1}+t^{n+2}+\cdots}{1+t^{2}+t^{3}+\cdots}
<\frac{t^{n+1}+t^{n+2}+\cdots}{t^{2}+t^{3}+\cdots}
= t^{n-1},
\end{equation*}
combining with \eqref{eq:Euler}, we obtain
\begin{equation*}
1>\prod_{t} \frac{1+t^{2}+t^{3}+\cdots+t^{n}}{1+t^{2}+t^{3}+\cdots}
>\prod_{t}\left( 1- t^{n-1}\right)
=\frac{1}{\zeta(n-1)}
\to 
1,\ \ \mathrm{as}\ n\to \infty
\end{equation*}
and complete the proof, where $\prod_t$ represents a product over all reciprocals $t$ of primes.

\smallskip
One can also show the right equality of \eqref{eq:limZ} using  the fact that
\begin{equation}\label{eq:increasing}
1<
1+p^{-2}+p^{-3}+\cdots+p^{-n}\, \uparrow\, 
1+p^{-2}+p^{-3}+\cdots,\ \ \mathrm{as}\ n\to \infty
\end{equation}
and the following version of monotone convergence theorem (which will also be very useful  later in proving Theorem \ref{thm:l}\,(iii)).
\begin{theorem}\label{thm:MCT}
If real numbers $x_{i,j}$ $(i,j=1,2,\dots)$  satisfy $1\le x_{i,j}\uparrow x_i$ as $j\to \infty$ for all $i$, then we have 
\begin{equation}\label{eq:MCT}
\lim_{j\to \infty}\prod_{i=1}^{\infty}x_{i,j}
=\prod_{i=1}^{\infty}x_{i}\,.
\end{equation}
Here we allow the products and the limit to be infinity.
\end{theorem}
\begin{proof}
Applying the monotone convergence theorem to $\log x_{i,j}\,(\ge 0)$ gives
$$\lim_{j\to \infty}\sum_{i=1}^{\infty}\log x_{i,j}
=\sum_{i=1}^{\infty}\log x_{i}\,.$$
Thus
$$\lim_{j\to \infty}\log \prod_{i=1}^{\infty}x_{i,j}
=\log \prod_{i=1}^{\infty}x_{i}\,,$$
and \eqref{eq:MCT} follows.
\end{proof}

Thanks to \eqref{eq:increasing}, we can apply Theorem \ref{thm:MCT} with $x_{i,j}=1+p_i^{-2}+p_i^{-3}+\cdots+p_i^{-j}$ and $x_i=1+p_i^{-2}+p_i^{-3}+\cdots$, and arrive at the desire equality.
\end{proof}

\begin{remark} 
(1) The proof of Theorem \ref{thm:cyclic SNF}\,(iv) is reminiscent of (though not directly related to) \cite[Exercise 1.186\,(c)]{S}.

(2) Theorem \ref{thm:cyclic SNF}\,(i), (iv) and the numerical value of (iii) are obtained in \cite[Section 3]{E} via a slightly different approach. 
We have provided a complete and more detailed  proof.
\end{remark}

\subsubsection{\texorpdfstring{More Generators (General $\ell$)}{More Generators (General l)}}\label{sec:l>1}
$ $

Now we consider the SNFs with at most $\ell\,(\le n)$ diagonal entries not equal to $1$, i.e., whose corresponding cokernel has at most $\ell$ generators. 
Denote the set of such $n\times n$ SNFs by $\mathcal{T}_n(\ell)$. 
In particular, when $\ell=n$, we have $\mu(\mathcal{T}_n(n))=1$.
The above discussion on cyclic SNFs is for the case $\ell=1$.
We will compute the density $\mu (\mathcal{T}_n(\ell))$ and its limit as $n\to \infty$\,, show that this limit increases to $1$ as $\ell\to \infty$\,, and establish its asymptotics.

\smallskip
We start with a lemma which will play an important role in our proof (as well as in Section \ref{sec:lim} below).

\begin{lemma}\label{lemma:Cp}
For any positive number $x\le 1/2$, the positive sequence $\{[1/x,k]\}_{k=1}^{\infty}$ is decreasing and thus has a limit as $k\to \infty$:
\begin{equation}\label{eq:C(x)}
C(x):=\left( 1-x\right) ( 1-x^{2}) \cdots \in \left.\left[  e^{-2x/(1-x)},1\right)\right..
\end{equation}
This also implies that  $C(x)\to 1$ as $x\to 0$ and that $[1/x,k]\in [e^{-2x/(1-x)},1)$ for all $x\in(0,1/2]$ and $k\ge 1$. 

In particular, when $x=1/p$, we have 
\begin{equation}\label{eq:C_p}
[p,k]\downarrow C_p:=C(1/p) 
\in \left.\left[ e^{-2/(p-1)},1\right)\right. \subseteq  \left.\left[ e^{-2},1\right)\right.,  \  \ {\mathrm{as}}\ \ k\to \infty\,,
\end{equation}
$C_p\to 1$ as $p\to \infty$\,, and $[p,k]\in [e^{-2/(p-1)},1)$ for all $p$ and $k\ge 1$. 
\end{lemma}
\begin{proof}
The sequence  $[1/x,k]$ is strictly decreasing in $k$ because  $0<1-x^{j}<1$ for all $j\ge 1$.

To get the lower bound for $C(x)$\,, we will use the following inequality:
\begin{equation}\label{eq:ln y}
\ln y \ge - \frac{1-y}{y}\,,\quad \forall\  y\in (0,1]\,.
\end{equation}
(To see this, let $\psi(y):=\ln y + (1-y)/y$, then $\psi'(y)=1/y-1/y^2\le 0$\,. Hence $\psi(y)\le \psi(1)=0$\,.)

Applying \eqref{eq:ln y} with $y=1-x^j$ $(j\ge 1)$  yields
\begin{equation}\label{eq:ln(1-x^j)}
\ln \left(1-x^j \right) 
\ge -\frac{x^j}{1-x^j}
\ge -2 x^j
\end{equation}
as $x^j\le 1/2$\,.
Summing up \eqref{eq:ln(1-x^j)} over $j$ from $1$ to $k$, we get
\begin{equation*}
\ln\, [1/x,k]
\ge -\sum_{j=1}^{k}2 x^j
>-\sum_{j=1}^{\infty}2 x^j
=-\frac{2x}{1-x}\,.
\end{equation*}
Hence $C(x)=\lim_{k\to \infty}[1/x,k] \ge e^{-2x/(1-x)}$.
\end{proof}

\begin{theorem}\label{thm:l}
We have

\noindent
{\rm (i)} 
\begin{equation}\label{eq:Z_n(l)}
\mu (\mathcal{T}_n(\ell)) 
=\prod_{p} Z_n(p,\ell)
= \frac{1}{\prod_{i=2}^{n}\zeta(i)}  \prod_{p}Y_n(p,\ell)=:Z_n(\ell)\,,
\end{equation} 
where
\begin{equation*} 
Z_n(p,\ell)=\mu_p (\mathcal{T}_n(\ell)) 
= [p,n] \sum_{i=0}^{\ell} \frac{p^{-i^2}[p,n]}{[p,i]^2 [p,n-i]}\,,
\end{equation*} 
\begin{equation}\label{eq:Y_n(p,l)}
Y_n(p,\ell)=\frac{[p,1]}{[p,n]} Z_n(p,\ell)
=[p,1]\sum_{i=0}^{\ell}\frac{p^{-i^2}[p,n]}{[p,i]^2 [p,n-i]}\,;
\end{equation} 

\noindent
{\rm (ii)} 
\begin{equation}\label{eq:Y(p,l)}
Y_n(p,\ell)\, \uparrow\, [p,1]\sum_{i=0}^{\ell}\frac{p^{-i^2}}{[p,i]^2}=:Y(p,\ell)  \ \ {\mathrm{as}}\ \  n\to \infty\,,\quad Y(p,\ell) \, \uparrow\, \frac{[p,1]}{C_p}  \ \ {\mathrm{as}}\  \ \ell\to \infty\,,
\end{equation} 
and
\begin{equation}\label{eq:Z(p,l)}
\mu_p (\mathcal{T}_n(\ell)) =Z_n(p,\ell)
\to \frac{C_p} {[p,1]} Y(p,\ell)
=C_p \sum_{i=0}^{\ell}\frac{p^{-i^2}}{[p,i]^2}
=: Z(p,\ell)  \ \  {\mathrm{as}}\ \ n\to \infty\,,
\end{equation} 
where $C_p=(1-p^{-1})(1-p^{-2})\cdots$ as defined in \eqref{eq:C_p} and \eqref{eq:C(x)}, 
then it follows from \eqref{eq:Y(p,l)} that
\begin{equation}\label{eq:lim Z(p,l)}
Z(p,\ell) \uparrow 1\ \ {\mathrm{as}}\ \ell\to \infty\,;
\end{equation} 

\noindent
{\rm (iii)} 
\begin{equation}\label{eq:Z(l)}
\mu (\mathcal{T}_n(\ell)) = Z_n(\ell)
 \to  \frac{1}{\prod_{i=2}^{\infty}\zeta(i)}  \prod_p Y(p,\ell) 
= \prod_p Z(p,\ell) 
=: Z(\ell)  \  \ {\mathrm{as}}\ \ n\to \infty\,,
\end{equation}  
and $Z(\ell) \uparrow 1$ as $\ell\to \infty$\,;

\noindent
{\rm (iv)} 
\begin{equation}\label{eq:asymZ(p,l)}
\lim_{n\to \infty} \mu_p (\mathcal{T}_n(\ell)) 
= Z(p,\ell)=1-C_p^{-1} p^{-(\ell+1)^2}\left[ 1- \frac{2}{p^2-p}\cdot p^{-\ell}+O\left( p^{-2\ell}\right) \right] \  \ {\mathrm{as}}\ \ \ell\to \infty\,;
\end{equation}
more precisely, this $O\left( p^{-2\ell}\right)\in(0,2p^{-2\ell})$\,; 

\noindent
{\rm (v)} 
\begin{equation}\label{eq:asymZ(l)}
\lim_{n\to \infty} \mu (\mathcal{T}_n(\ell)) 
= Z(\ell)=1-C_2^{-1}\cdot 2^{-(\ell+1)^2}\left[ 1- 2^{-\ell}+O\left( 4^{-\ell}\right) \right]\  \ {\mathrm{as}}\ \ \ell\to \infty\,,
\end{equation}
where $ C_2^{-1}\approx3.46275\,.$

\smallskip
Parts {\rm (ii)}  and {\rm (iv)}  also hold with $p=1/x$ for any $x\in (0,1/2]$.
\end{theorem}

Figure \ref{fig:1} and Table \ref{table:1} below illustrate the asymptotics \eqref{eq:asymZ(l)} of $Z(\ell)$ and fast rate of convergence.

\begin{figure} [!h]
\caption[]{Asymptotics of $Z(\ell)$}\label{fig:1}
\vspace{0.2cm}
\centering
\begin{tikzpicture}
\begin{axis}[
width=3in, height=3in,
axis x line=middle,axis y line=middle,
xlabel = $\ell$,
ylabel = $ 2^{(\ell+1)^2}(1-Z(\ell))$,
xmax = 22,xmin = 1,
ymax = 3.7,ymin = 2.4
]
\addplot[
color = black, fill = black,
mark = *, only marks
]  coordinates{
(1, 2.44902557224)
(2,2.75103562616)
(3, 3.06085395424)
(4, 3.25359037644)
(5, 3.35635172814)
(6, 3.40909705378)
(7, 3.43580813230)
(8, 3.44924885316)
(9, 3.45599059345)
(10, 3.45936681921)
(11, 3.46105627233)
(12, 3.46190133412)
(13, 3.46232394884)
(14, 3.46253527716)
(15, 3.46264094656)
(16, 3.46269378257)
(17, 3.46272020090)
(18, 3.46273341015)
(19, 3.46274001480)
(20, 3.46274331712)
};
\end{axis}
\end{tikzpicture}
\hspace{1cm}
\centering
\begin{tikzpicture}
\begin{axis}[
width=3in, height=3in,
axis x line=middle,axis y line=middle,
xlabel = $\ell$,
ylabel = $-\!\ln (1- C_2\, 2^{(\ell+1)^2}(1\!-\!Z(\ell)))\, /\ln 2$,
xmax = 24.5,xmin = 1,
ymax = 24.5,ymin = 1
]
\addplot[
color = black, fill = black,
mark = *, only marks
]  coordinates{
(1, 1.77225611430)
(2, 2.28255339912)
(3, 3.06085395424)
(4, 4.04926385851)
(5, 5.02441603986)
(6, 6.01220652280)
(7, 7.00610418193)
(8, 8.00305233425)
(9, 9.00152622794)
(10, 10.0007631292)
(11, 11.0003815684)
(12, 12.0001907852)
(13, 13.0000953928)
(14, 14.0000476965)
(15, 15.0000238483)
(16, 16.0000119242)
(17, 17.0000059621)
(18, 18.0000029811)
(19, 19.0000014906)
(20, 20.0000007454)
};
\end{axis}
\end{tikzpicture}

\end{figure}

\begin{table}[!h]
\caption{Asymptotics of $Z(\ell)$}\label{table:1}
\begin{tabular}{@{} rclcc @{}}
\hline \rule{0pt}{0.5cm}
$\ell$ & \multicolumn{1}{c}{$Z(\ell)$} & \multicolumn{1}{c}{$1-Z(\ell)$} & \multicolumn{1}{c}{$2^{(\ell+1)^2}(1\!-\!Z(\ell))$} & \multicolumn{1}{c}{$-\!\ln [1- C_2\, 2^{(\ell+1)^2}(1\!-\!Z(\ell))]\,/\ln 2$} \vspace{0.05cm}\\
\hline \rule{0pt}{0.5cm}
$1$ & 0.846935901735 & $1.53064098265\times 10^{-1}$ & 2.44902557224 & 1.77225611430  \\
$2$ & 0.994626883543 & $5.37311645734\times 10^{-3}$ & 2.75103562616 & 2.28255339912  \\
$3$ & 0.999953295075 & $4.67049248389\times 10^{-5}$ & 3.06085395424 & 3.10703467197 \\
$4$ & 0.999999903035 & $9.69645493161\times 10^{-8}$ & 3.25359037644 & 4.04926385851  \\
$5$ & 0.999999999951 & $4.88413458245\times 10^{-11}$ & 3.35635172814 & 5.02441603986  \\
$6$ & 1.000000000000 & $6.05577286766\times 10^{-15}$ & 3.40909705378 & 6.01220652280 \\
$7$ & 1.000000000000 & $1.86255532064\times 10^{-19}$ & 3.43580813230 & 7.00610418193  \\
$8$ & 1.000000000000 & $1.42657588960\times 10^{-24}$ & 3.44924885316 & 8.00305233425  \\
$9$ & 1.000000000000 & $2.72629586798\times 10^{-30}$ & 3.45599059345 & 9.00152622794
\\
${10}$ & 1.000000000000 & $1.30126916909\times 10^{-36}$ & 3.45936681921 & 10.0007631292
\vspace{0.05cm}
\\ \hline
\end{tabular}
\end{table}

\begin{remark} 
The convergence result \eqref{eq:lim Z(p,l)} in Theorem \ref{thm:l}\,(ii) with $p=1/x$ implies Euler's identity:
\begin{equation*}
\sum_{i=0}^{\infty}\frac{x^{i^2}}{(1-x)^2(1-x^2)^2\cdots(1-x^i)^2}
=\frac{1}{(1-x)(1-x^2)\cdots}\,.
\end{equation*}
\end{remark}

\begin{proof} 
(i)
The first equality follows from Theorem \ref{thm:mu(S)} with $\mathcal{S}=\mathcal{T}_n(\ell)$\,, $r=n-\ell$\,, $d_i=z=1$\,, $s_j=0$ for all $i,j$, 
and Theorem \ref{thm:p^s} with  $s=1$\,,  $a_1=n,n-1,\dots,n-\ell$\,, respectively.

The second equality follows from definition \eqref{eq:Y_n(p,l)} and \eqref{eq:zeta}.

\smallskip
(ii) 
We observe that 
\begin{equation*} 
\frac{[p,n]}{[p,n-i]} 
= \left( 1-p^{-n}\right) \left( 1-p^{-(n-1)}\right)  \cdots\left( 1-p^{-(n-i+1)}\right) \, \uparrow\, 1\ \ {\mathrm{as}}\ \ n\to \infty\,.
\end{equation*} 
This leads to the first result of \eqref{eq:Y(p,l)}.

Since $Y_n(p,\ell)$ is also increasing in $\ell$ by definition \eqref{eq:Y_n(p,l)}, so is $Y(p,\ell)$, and for all $\ell\le n$, we have
\begin{equation}\label{eq:Y_l(p,l)}
Y_{\ell}(p,\ell)\le Y_n(p,\ell)\le Y_n(p,n)\,.
\end{equation}
Further, we derive from
$$1=\mu_p(\mathcal{T}_n(n)) =\frac{[p,n]}{[p,1]} Y_n(p,n)\,,$$
that 
\begin{equation*}
Y_n(p,n)=\frac{[p,1]}{[p,n]}\quad {\mathrm{and\ similarly,}}\quad Y_{\ell}(p,{\ell})=\frac{[p,1]}{[p,\ell]}\,.
\end{equation*}
Plugging into \eqref{eq:Y_l(p,l)}, we obtain
\begin{equation*}
\frac{[p,1]}{[p,\ell]}\le Y_n(p,\ell)\le \frac{[p,1]}{[p,n]}<\frac{[p,1]}{C_p}\,.
\end{equation*}
Taking $n\to \infty$ yields
\begin{equation*} 
\frac{[p,1]}{[p,\ell]}\le Y(p,\ell)\le \frac{[p,1]}{C_p}\,.
\end{equation*}
Then taking $\ell\to \infty$ and applying Lemma \ref{lemma:Cp} leads to the second result of \eqref{eq:Y(p,l)}.

Finally, on the strength of \eqref{eq:Y(p,l)} and Lemma \ref{lemma:Cp}, we obtain \eqref{eq:Z(p,l)} from  definition \eqref{eq:Y_n(p,l)}:
\begin{equation*} 
Z_n(p,\ell) = \frac{[p,n]} {[p,1]} Y_n(p,\ell)
\to \frac{C_p} {[p,1]} Y(p,\ell)  \ \  {\mathrm{as}}\ \ n\to \infty\,.
\end{equation*}

This proof also carries over to $p=1/x$ for any $x\in (0,1/2]$.

\smallskip
(iii) 
It follows from definitions \eqref{eq:Z_n(l)} and  \eqref{eq:Y_n(p,l)} that
\begin{equation} \label{eq:prod Z_n(p,l)}
Z_n(\ell) = \prod_p Z_n(p,\ell) 
=\prod_p \frac{[p,n]}{[p,1]} \prod_p Y_n(p,\ell) \,.
\end{equation} 

By virtue of \eqref{eq:zeta}, we have
\begin{equation} \label{eq:prod1}
\prod_p \frac{[p,n]}{[p,1]}=\frac{1}{\prod_{i=2}^n \zeta(i)}
\to \frac{1}{\prod_{i=2}^{\infty} \zeta(i)} \approx 0.435757
\ \  {\mathrm{as}}\ \ n\to \infty\,.
\end{equation} 
Further, this limit
\begin{equation*}
\frac{1}{\prod_{i=2}^{\infty} \zeta(i)} 
=\prod_{i=2}^{\infty}  \prod_p \left(1-p^{-i} \right) 
= \prod_p \prod_{i=2}^{\infty}  \left(1-p^{-i} \right) 
=\prod_p \frac{C_p}{[p,1]}\,.
\end{equation*} 
Hence
\begin{equation} \label{eq:prod1'}
\prod_p \frac{[p,n]}{[p,1]}
\to \prod_p \frac{C_p}{[p,1]}  \ \ {\mathrm{as}}\ \ n\to \infty\,.
\end{equation}
We can also deduce \eqref{eq:prod1'} from Theorem \ref{thm:MCT} with $x_{i,j}=[p_i,1]/[p_i,j]$ 
since 
$$1\le \frac{[p,1]}{[p,n]}\,\uparrow\,  \frac{[p,1]}{C_p} \ \ {\mathrm{as}}\ \ n\to \infty$$ 
by Lemma \ref{lemma:Cp}.

For the second product on the right-hand side of \eqref{eq:prod Z_n(p,l)}, from \eqref{eq:Z_nPf} in the proof of Theorem \ref{thm:cyclic SNF}\,(i), we see that $Y_n(p,1)=1+p^{-2}+p^{-3}+\cdots+p^{-n}>1$\,. 
Since $Y_n(p,\ell)$ is increasing in $\ell$, we have $Y_n(p,\ell)>1$ as well.
In conjunction with \eqref{eq:Y(p,l)}, 
we can apply Theorem \ref{thm:MCT} with $x_{i,j}=Y_j(p_i,\ell)$ to obtain
\begin{equation} \label{eq:prod2}
\prod_p Y_n(p,\ell)\,
\uparrow\, \prod_p Y(p,\ell)  \  \ {\mathrm{as}}\ \ n\to \infty\,.
\end{equation} 
Plugging \eqref{eq:prod1}, \eqref{eq:prod2} and \eqref{eq:prod1'} into \eqref{eq:prod Z_n(p,l)}
 along with definition \eqref{eq:Z(p,l)} yields \eqref{eq:Z(l)}:
\begin{equation}\label{eq:lim Z_n(l)}
Z_n(\ell)
\to  \frac{1}{\prod_{i=2}^{\infty}\zeta(i)}  \prod_p Y(p,\ell) 
= \prod_p \frac{C_p}{[p,1]} \prod_p Y(p,\ell)
= \prod_p \frac{C_p}{[p,1]} Y(p,\ell)
= \prod_p Z(p,\ell)
\end{equation} 
as $n\to \infty$\,.

Since $Y_n(p,\ell)>1$ and $Y_n(p,\ell)$ is increasing in $\ell$, so is $Y(p,\ell)$ (recall \eqref{eq:Y(p,l)}). 
Thus we can apply Theorem \ref{thm:MCT} with $x_{i,j}=Y(p_i,j)$ to obtain
\begin{equation*} 
\prod_p Y(p,\ell)\,
\uparrow\, \prod_p \frac{[p,1]}{C_p} \  \ {\mathrm{as}}\ \ \ell\to \infty\,.
\end{equation*} 
Finally, we plug this into  the second expression of the limit of $Z_n(\ell)$ in \eqref{eq:lim Z_n(l)}:
\begin{equation*} 
Z(\ell)  =\lim_{n\to \infty} Z_n(\ell)
=\prod_p \frac{C_p}{[p,1]} \prod_p Y(p,\ell)\,
\uparrow\, 1 \  \ {\mathrm{as}}\ \ \ell\to \infty\,.
\end{equation*} 

\smallskip
(iv) We prove for the more general case $p=1/x$ with $x\in (0,1/2]$.
Let 
$$V(x,\ell):=Z(1/x,\ell)=C(x) \sum_{i=0}^{\ell}\frac{x^{i^2}}{[1/x,i]^2}\,.$$
Recall that $C(x)=(1-x)(1-x^2)\cdots$ and $[1/x,i]=(1-x)(1-x^2)\cdots(1-x^i)$.

Since $V(x,\ell)=Z(1/x,\ell)\uparrow 1$ as $\ell\to \infty$ by \eqref{eq:lim Z(p,l)}, we have
$$\frac{1}{C(x)}
=\sum_{i=0}^{\infty}\frac{x^{i^2}}{[1/x,i]^2}
=\sum_{i=0}^{\ell}\frac{x^{i^2}}{[1/x,i]^2}
+\sum_{i=\ell+1}^{\infty}\frac{x^{i^2}}{[1/x,i]^2}
=\frac{V(x,\ell)}{C(x)}+\sum_{i=\ell+1}^{\infty}\frac{x^{i^2}}{[1/x,i]^2}\,.$$
Thus for any $x\in(0,1/2]$, we obtain
\begin{align}  \nonumber
&\  x^{-(\ell+1)^2}C(x)\left[1-V(x,\ell)\right] 
=x^{-(\ell+1)^2}C^2(x) \left[\frac{1}{C(x)}-\frac{V(x,\ell)}{C(x)} \right]  =x^{-(\ell+1)^2}  \sum_{i=\ell+1}^{\infty}\frac{C^2(x)x^{i^2}}{[1/x,i]^2}\\ \nonumber
=&\  \sum_{i=\ell+1}^{\infty}x^{i^2-(\ell+1)^2} \prod_{j=i+1}^{\infty}\left(1-x^j\right)^2
=\prod_{j=\ell+2}^{\infty}\left(1-x^j\right)^2
+\sum_{i=\ell+2}^{\infty}x^{i^2-(\ell+1)^2} \prod_{j=i+1}^{\infty}\left(1-x^j\right)^2\\ \label{eq:V(x,l)}
=&\ \left( 1-2\sum_{j=\ell+2}^{\infty}x^j+\Delta_1\right) 
+\Delta_2
= 1-\frac{2x^{\ell+2} }{1-x}+\Delta_1
+\Delta_2\,,
\end{align}
where
\begin{equation}\label{eq:Delta2}
0< \Delta_2:=\sum_{i=\ell+2}^{\infty}x^{i^2-(\ell+1)^2} \prod_{j=i+1}^{\infty}\left(1-x^j\right)^2
<\sum_{i=\ell+2}^{\infty}x^{i^2-(\ell+1)^2} 
<\sum_{i=2\ell+3}^{\infty}x^{i} =\frac{x^{2\ell+3} }{1-x}<x^{2\ell} 
\end{equation}
and
\begin{equation}\label{eq:Delta1}
0\le \Delta_1:=\prod_{j=\ell+2}^{\infty}\left(1-x^j\right)^2-\left( 1-2\sum_{j=\ell+2}^{\infty}x^j\right) 
\le 4\sum_{j,j'\ge \ell+2}x^{j+j'}
=\frac{4x^{2\ell+4} }{(1-x)^2 }
\le  x^{2\ell}\,,
\end{equation}
as $0<x\le 1/2$\,, thanks to the inequality:
\begin{equation*} 
0\le \prod_{i=1}^{u}(1-\delta_i)
-\left(  1-\sum_{i=1}^{u}\delta_i\right)
\le \sum_{1\le i<j\le u}\delta_i \delta_j \end{equation*}
for $\delta_1,\delta_2,\dots,\delta_{u}\in [0,1]$, which can be proved easily by induction on $u$ (the left inequality was proved in \eqref{eq:1-delta}. 
For the right inequality, base cases: $u=1,2$; inductive step from $u$ to $u+1$: $(1-\delta_{u+1}) \prod_{i=1}^{u}(1-\delta_i)
\le (1-\delta_{u+1}) (1-\sum_{i=1}^{u}\delta_i+\sum_{1\le i<j\le u}\delta_i \delta_j ) 
=1-\sum_{i=1}^{u+1}\delta_i+\sum_{1\le i<j\le u+1}\delta_i \delta_j-\delta_{u+1}\sum_{1\le i<j\le u}\delta_i \delta_j
\le 1-\sum_{i=1}^{u+1}\delta_i+\sum_{1\le i<j\le u+1}\delta_i \delta_j$).

\smallskip
Plugging \eqref{eq:Delta2} and \eqref{eq:Delta1} into \eqref{eq:V(x,l)} yields \eqref{eq:asymZ(p,l)}.

\smallskip
(v)
Since $Z(\ell)=\prod_p Z(p,\ell)$ by definition \eqref{eq:Z(l)} and $0\le Z(p,\ell)\le 1$ for all $p$\,, we  have $Z(\ell)\le Z(2,\ell)$. Thus it follows from (iv) that
\begin{equation}\label{eq:Z(l)<}
Z(\ell)\le Z(2,\ell)
=1-C_2^{-1} 2^{-(\ell+1)^2}\left[ 1- 2^{-\ell}+O\left( 4^{-\ell}\right) \right] \  \ {\mathrm{as}}\ \ \ell\to \infty\,.
\end{equation}

On the other hand, we notice that when $\ell\ge 2$, the $O(p^{-2\ell})$ in \eqref{eq:asymZ(p,l)} satisfies 
$$O(p^{-2\ell})<2p^{-2\ell}\le \frac{2}{p^2}\cdot p^{-\ell}<\frac{2}{p^2-p}\cdot p^{-\ell}\,,$$
thus
\begin{equation*}
Z(p,\ell) > 1-C_p^{-1} p^{-(\ell+1)^2}\,.
\end{equation*}
Hence
\begin{equation}\label{eq:Z(l)>}
Z(\ell)=\prod_p Z(p,\ell)>Z(2,\ell) \prod_{p\ge 3} \left( 1-C_p^{-1} p^{-(\ell+1)^2}\right)
\ge 1 - (1-Z(2,\ell)) - \sum_{p\ge 3} C_p^{-1} p^{-(\ell+1)^2}.
\end{equation}
Here we took advantage of the inequality \eqref{eq:1-delta}. 
Thanks to \eqref{eq:C_p}, the positive sum
\begin{align*}
\sum_{p\ge 3} C_p^{-1} p^{-(\ell+1)^2}
\le&\  e^2 \sum_{p\ge 3} p^{-(\ell+1)^2}
=e^2 \,2^{-(\ell+1)^2}\sum_{p\ge 3} \left( \frac{2}{\,p\,}\right)^{(\ell+1)^2}
<\ e^2 \,2^{-(\ell+1)^2}\sum_{p\ge 3} \left( \frac{2}{\,3\,}\right)^{\ell^2}\left( \frac{2}{\,p\,}\right)^{2}\\
<&\
e^2 \,2^{-(\ell+1)^2}\left( \frac{2}{\,3\,}\right)^{\ell^2}\cdot 4
=2^{-(\ell+1)^2}O(4^{-\ell})\,.
\end{align*}
Finally, combining with \eqref{eq:Z(l)<} and \eqref{eq:Z(l)>} leads to \eqref{eq:asymZ(l)}.
\end{proof}

\begin{remark}  
When $\ell=1$, in the proof of Theorem \ref{thm:cyclic SNF} we wrote $Y(1/x,1)$ as $(1-x^{6})/(1-x^{2})(1-x^{3})$ (see \eqref{eq:Y(p,1)}) in order to represent 
$Z(1)= \prod_p Y(p,1)\, / \prod_{i=2}^{\infty}\zeta(i)$ 
as the reciprocal of a product of values of the Riemann zeta function at positive integers. However, this is not the case when $\ell>1$; 
in fact, in general $Y(1/x,\ell)$ is not even a symmetric function in $x$, for instance,
$$Y\left( \frac{1}{\,x\,},2\right) 
=\frac{1 - x - x^2 + 2 x^3 - x^5 + x^6}{(1-x)^3 (1+x)^2}\,,$$
$$Y\left( \frac{1}{\,x\,},3\right) 
=\frac{1 - x - x^2 + 2 x^4 + x^5 - 2 x^6 - x^7 + x^8 + x^9 - x^{11} + x^{12}}{(1 - x)^5 (1 + x)^2 (1 + x + x^2)^2}\,.$$
\end{remark}

\section{\texorpdfstring{Properties of the SNF Distribution Function $\mu_{p^s}$}{Properties of the SNF Distribution Function mu{\_}\{p{\^{}}s\}}}
\label{sec:properties}
In this section, we first fix $p,s,m,n$ and find the maximum and minimum of the probability density function $\mu_{p^s}$ of \eqref{eq:mu_{p^s}(D)}.
Then we free $p,s,m,n$ and study the monotonicity properties and limiting behaviors of  $\mu_{p^s}(\{D_{\bm{a}}\})$, as a function of $p,s,m,n$ and ${\bm a}$ (recall from Theorem \ref{thm:p^s} the notation of vector ${\bm a}=(a_1,a_2,\dots,a_s)$ as well as its corresponding diagonal matrix $D_{\bm a}\in \mathbb{S}$). 

For convenience, we replace $m-a_i$ by $b_i$ $(0\le i\le s)$ in \eqref{eq:mu_{p^s}(D)} to get a simpler expression for $\mu_{p^s}(\{D_{\bm{a}}\})$:
\begin{equation}\label{eq:f}
f(p,s,m,n',{\bm b})
:=p^{- \sum_{i=1}^{s}(n'+b_i)b_i}\cdot
\frac{[p,n'+m][p,m]}
{[p,n'+b_s][p,b_s] \prod_{i=1}^{s} [p,b_{i-1}-b_i]}\,.
\end{equation}
Here and throughout this section, we shall assume that $p$ is a prime, that $s,m$ and $n$ are positive integers, that $n>n':=n-m\ge 0$\,, and that ${\bm b}:=(b_1,b_2,\dots,b_s)$ is an integer vector satisfying
$m=b_0\ge b_1\ge \cdots \ge b_s \ge 0$\,.

\subsection{The Maximum and  Minimum}\label{sec:max}
$ $

We show that  $f(p,s,m,n',\cdot)$ attains its maximum at either $(0,0,\dots,0)$ or $(1,1,\dots,1)$ depending on $p,s,m$ and $n'$, and its minimum at $(m,m,\dots,m)$.

\begin{theorem}\label{thm:max}
For fixed $p,m,n$ and $s$, the maximum and minimum of $f(p,s,m,n',\cdot)$ are given as follows.

\smallskip
{\rm (i)} If $p>2$\,, $s>1$ or $n'>0$\,, then
\begin{equation*}
\max_b \ f(p,s,m,n',{\bm b})
=\frac{[p,n'+m]}{[p,n']} \,,
\end{equation*}
and the maximum is achieved if and only if ${\bm b}=(0,0,\dots,0):={\bm 0}$, in other words, if the corresponding matrix $D_{\bm a}$ is full rank;

\smallskip
{\rm (ii)} If $p=2$\,, $s=1$\,, $n'=0$  and $m>1$\,, then
\begin{equation*}
\max_b \ f(p,s,m,n',{\bm b})
=\frac{[2,m]^2}{[2,1][2,m-1]}\,,
\end{equation*}
and the maximum is achieved if and only if \,${\bm b}=(1)$;

\smallskip
{\rm (iii)}  In both Case {\rm (i)} and Case {\rm (ii)}, we have
\begin{equation*}
\min_b \ f(p,s,m,n',{\bm b})
=p^{-s(n'+m)m}\,,
\end{equation*}
and the minimum is achieved if and only if
\,${\bm b}=(m,m,\dots,m)$, in other words, if the corresponding matrix $D_{\bm a}$ is the zero matrix.

\smallskip
{\rm (iv)} If $p=2$\,, $s=1$\,, $n'=0$  and $m=1$\,, then ${\bm b}=(1)$ or $(0)$, and they have the same value of $f$: $1/2$.
\end{theorem}
\begin{proof}
(i) We proceed by the following two lemmas which show that the $b_i$'s are all equal  at the maximum of $f(p,s,m,n',\cdot)$, and that $b_i=0$ or $1$ depending on $p,s,m$ and $n'$\,. 

Let ${\bm b}=(b_1,b_2,\dots,b_s)$ be an arbitrary $s$-tuple with $m=b_0\ge b_1\ge \cdots \ge b_s \ge 0$\,.

\begin{lemma}\label{lemma:b_i equal}
If $b_i>b_{i+1}$ for some $i\in \{1,2,\dots,s-1\}$ $(s\ge 2)$, then we have
\begin{equation*}
f(p,s,m,n',{\bm b}')> f(p,s,m,n',{\bm b}),
\end{equation*}
where ${\bm b}'=(b'_1,b'_2,\dots,b'_s)$ with $b'_i=b_i-1$ and $b'_j=b_j$ for all $j\neq i$. Note that ${\bm b}'$ still satisfies $m=b'_0\ge b'_1\ge \cdots \ge b'_s \ge 0$\,.
\end{lemma}

\begin{lemma}\label{lemma:v}
Let $\varphi(b):=f\big( p,s,m,n',(b,b,\dots,b)\big)$,  $0\le b\le m$, then for all $0\le b<m$, we have 
\begin{equation*}
\frac{\varphi(b)}{\varphi(b+1)}
\begin{cases} 
<1\,, & \mathrm{if} \ p=2\,,\ s=1\,,\ n'=0\,,\ m>1\,\  \mathrm{and}\,\ b=0 \\ 
=1\,, & \mathrm{if}\ p=2\,,\ s=1\,,\ n'=0\,,\ m=1\,\  \mathrm{and}\,\ b=0\\
>1\,, & \mathrm{otherwise} 
\end{cases} .
\end{equation*}
\end{lemma}

These lemmas are proved right below this proof. 
Thanks to Lemma \ref{lemma:b_i equal}, the maximum point of $f(p,s,m,n',\cdot)$ must have the form $(b,b,\dots,b)$ with $0\le b \le m$\,. Therefore it reduces to finding the maximum of $\varphi(\cdot)$.

Since $p>2$\,, $s>1$ or $n'>0$\,, it follows from Lemma \ref{lemma:v} that 
\begin{equation}\label{eq:v0}
\varphi(0)>\varphi(1)>\cdots>\varphi(m)\,.\end{equation}
Hence the maximum of $\varphi(\cdot)$ is $\varphi(0)=\frac{[p,n'+m]}{[p,n']}$\,, as desired. 

\smallskip
(ii) When $p=2$\,, $s=1$\,, $n'=0$  and $m>1$\,, it follows from Lemma \ref{lemma:v} that 
\begin{equation}\label{eq:v}
\varphi(0)<\varphi(1) \quad {\rm{and}} \quad \varphi(1)>\cdots>\varphi(m)\,.
\end{equation}
Hence the maximum of $\varphi(\cdot)$ is $\varphi(1)=\frac{[2,m]^2}{[2,1][2,m-1]}$\,, as desired.

\smallskip
(iii) We proceed by the following lemma (proved right below this proof) which shows that at the minimum of $f(p,s,m,n',\cdot)$, all the $b_i$'s $(i>1)$ equal   $m$.

\begin{lemma}\label{lemma:b_i=m}
If $b_i<b_{i-1}$ for some $i\in \{1,2,\dots,s-1\}$ $(s\ge 2)$, then we have
\begin{equation*}
f(p,s,m,n',{\bm b}')< f(p,s,m,n',{\bm b}),
\end{equation*}
where ${\bm b}'=(b'_1,b'_2,\dots,b'_s)$ with $b'_i=b_i+1$ and $b'_j=b_j$ for all $j\neq i$. Note that ${\bm b}'$ still satisfies $m=b'_0\ge b'_1\ge \cdots \ge b'_s \ge 0$\,.
\end{lemma}

Thanks to Lemma \ref{lemma:b_i=m}, the minimum point of $f(p,s,m,n',\cdot)$ must have the form $(m,m,\dots,m,b)$ with $0\le b \le m$\,. Further, since $f\big(p,s,m,n',(m,m,\dots,m,b)\big)
=p^{-(s-1)(n'+m)m}\cdot \varphi(b)$ (by \eqref{eq:f}),  
where $\varphi$ is defined in Lemma \ref{lemma:v} with $s=1$\,, it reduces to finding the minimum of $\varphi(\cdot)$.

\smallskip
{\it Case} (i) 
When $p>2$ or $n'>0$\,, it follows from \eqref{eq:v0} that the minimum of $\varphi(\cdot)$ is $\varphi(m)=p^{-(n'+m)m}$\,. 

{\it Case} (ii) 
When $p=2$\,, $n'=0$ and $m>1$\,, it follows from \eqref{eq:v} that  the minimum of ${\varphi}(\cdot)$ is $\min\,\{{\varphi}(0),{\varphi}(m)\}$. 
Since
$${\varphi}(0)=[2,m]>\left(1-2^{-1} \right)^m=2^{-m}\ge 2^{-m^2}=\varphi(m)\,,$$
the minimum of $\varphi(\cdot)$ is still ${\varphi}(m)$.

\smallskip
Hence the minimum of $f$ is always $p^{-s(n'+m)m}$ and achieved at $(m,m,\dots,m)$.
\end{proof}

\begin{proof}[Proof of Lemma \ref{lemma:b_i equal}]
It follows from definition \eqref{eq:f} that
\begin{align*}
&\frac{f(p,s,m,n',{\bm b}')}{f(p,s,m,n',{\bm b})}
=  p^{(n'+b_i)b_i-(n'+b'_i)b'_i}\cdot
\frac{ [p,b_{i-1}-b_i][p,b_i-b_{i+1}]}
{ [p,b_{i-1}-b'_i][p,b'_i-b_{i+1}]}\\
\ge&\ p\cdot
\frac{ [p,b_{i-1}-b_i][p,b_i-b_{i+1}]}
{ [p,b_{i-1}-b_i+1][p,b_i-1-b_{i+1}]}
= \ p\cdot
\frac{ 1-p^{-(b_i-b_{i+1})}}
{ 1-p^{-(b_{i-1}-b_i+1)}}
> p\left(1-p^{-1} \right) 
=p-1\ge 1\,,
\end{align*}
as desired, where in the second last inequality, we used the condition that $b_i>b_{i+1}$ to get
$1-p^{-(b_i-b_{i+1})}\ge 1-p^{-1}\,.$
\end{proof}

\begin{proof}[Proof of Lemma \ref{lemma:v}]
By the definition of $\varphi$ and \eqref{eq:f}, we obtain
\begin{align}
\nonumber
&\frac{\varphi(b)}{\varphi(b+1)}
= p^{s[(n'+b+1)(b+1)-(n'+b)b]}\cdot
\frac{[p,n'+b+1][p,b+1][p,m-b-1]}
{[p,n'+b][p,b][p,m-b]}\\
\label{eq:lemma:v}
=&\ p^{s(n'+2b+1)}\cdot
\frac{\left( 1-p^{-(n'+b+1)}\right)
\left( 1-p^{-(b+1)}\right)}
{ 1-p^{-(m-b)}}
> p^{s(n'+2b+1)}
\left( 1-p^{-1}\right)^2\,,
\end{align}
where we used the fact that 
$$1-p^{-(n'+b+1)},\ 1-p^{-(b+1)}\ge 1-p^{-1}\quad \text{and}\quad 1-p^{-(m-b)}<1\,.$$

\noindent
{\it Case 1.} 
$s(n'+2b+1)\ge 2$\,. 

The right-hand side of \eqref{eq:lemma:v} is at least
$$p^2\left( 1-p^{-1}\right)^2 = (p-1)^2\ge 1\,.$$

\noindent
{\it Case 2.} 
$p\ge 3$\,. 

The right-hand side of \eqref{eq:lemma:v} is at least
$$p\left( 1-p^{-1}\right)^2 \ge 3\left( 1-3^{-1}\right)^2 =4/3>1\,.$$

\noindent
{\it Case 3.} 
$p=2$ and $s(n'+2b+1)=1$\,, which requires that $s=1$ and $n'=b=0$\,. 

Plugging into \eqref{eq:lemma:v} yields
$$\frac{\varphi(b)}{\varphi(b+1)}
=\frac{2\left( 1-2^{-1}\right)^2}{1-2^{-m}}
=\frac{1}{2-2^{1-m}}
\begin{cases} 
<1, & \mbox{if }\ m>1 \\ 
=1, & \mbox{if }\ m=1
\end{cases} $$
and completes the proof.
\end{proof}

\begin{proof}[Proof of Lemma \ref{lemma:b_i=m}]
By definition \eqref{eq:f}, we obtain
\begin{align*}
&\frac{f(p,s,m,n',{\bm b}')}{f(p,s,m,n',{\bm b})}
=  p^{(n'+b_i)b_i-(n'+b'_i)b'_i}\cdot
\frac{ [p,b_{i-1}-b_i][p,b_i-b_{i+1}]}
{ [p,b_{i-1}-b'_i][p,b'_i-b_{i+1}]}\\
\le&\ p^{-1} \cdot
\frac{ [p,b_{i-1}-b_i][p,b_i-b_{i+1}]}
{ [p,b_{i-1}-b_i-1][p,b_i+1-b_{i+1}]}
= \ p^{-1} \cdot
\frac{ 1-p^{-(b_{i-1}-b_i)}}
{ 1-p^{-(b_i+1-b_{i+1})}}
< p^{-1}  \cdot \frac{1}{1-p^{-1} }
= \frac{1}{p-1}\le 1\,,
\end{align*}
as desired, where in the second last inequality, we used the condition that $b_i\ge b_{i+1}$ to get
$1-p^{-(b_i+1-b_{i+1})}\ge 1-p^{-1}\,.$
\end{proof}

\subsection{Monotonicity Properties and Limiting Behaviors}\label{sec:lim}
$ $

Now we free $p,s,m$ and $n'$. We will see that the monotonicity properties and limiting behaviors of $f$ of \eqref{eq:f} when ${\bm b}= {\bm 0}$ (i.e.,  the corresponding matrix $D_{\bm a}$ is full rank) differ tremendously from those when ${\bm b}\neq {\bm 0}$.
Specifically, we show that $f$ is increasing in $n',p$ and decreasing in $m$ when ${\bm b}= {\bm 0}$ (Theorem \ref{thm:b=0}), but decreasing in $n'$ and increasing in $m$ when ${\bm b}\neq {\bm 0}$ (Theorem \ref{thm:b!=0}).
Further, with regard to limiting behaviors, when ${\bm b}= {\bm 0}$, the limit of $f$ as $p,m$ or $n'\to \infty$ is positive (note that $f$ is independent of $s$)  (Theorem \ref{thm:b=0}); 
whereas when ${\bm b}\neq {\bm 0}$, the limit of $f$ is still positive as $m\to \infty$ or $s\to \infty$ with $\sum_{i=1}^s b_i$ bounded  (Theorems \ref{thm:non0lim}, \ref{thm:non0lim'}), 
but zero as $\max\left\{p,n',\sum_{i=1}^{s}b_i\right\}\to \infty$  (Theorems \ref{thm:lim0}, \ref{thm:lim0'}).
Lemma \ref{lemma:Cp} is crucial in the analysis of limiting behaviors of  $f$.

\subsubsection{\texorpdfstring{The Case of\, ${\bm b}= {\bm 0}$\,}{The Case of b=0}}
$ $ 

Let
\begin{equation}\label{eq:f0}
f_0(p,m,n'):=f(p,s,m,n',{\bm 0})=\frac{[p,n'+m]}{[p,n']}=\prod_{j=n'+1}^{n'+m}\left(1-p^{-j}\right).
\end{equation} 
We derive the following monotonicity properties and limiting behaviors of $f_0$ with the help of Lemma \ref{lemma:Cp}.

\begin{theorem}\label{thm:b=0}
The function $f_0(p,m,n')$ of \eqref{eq:f0}
is strictly increasing in $p,n'$ while strictly decreasing in $m$, and satisfies
$$ \lim_{m\to \infty} f_0(p,m,n')
=\frac{C_p}{[p,n']}<1\,, \quad 
\lim_{n'\to \infty} \inf_{m} f_0(p,m,n')=1  \quad  {\rm{and}}  \quad 
\lim_{p\to \infty} \inf_{m,n'} f_0(p,m,n')= 1 \,,$$
where $C_p=(1-p^{-1})(1-p^{-2})\cdots$ as defined in Lemma \ref{lemma:Cp}. In particular, we have
$$ \lim_{m\to \infty} f_0(p,m,0)
=C_p   \quad  {\rm{and}}  \quad 
\lim_{n'\to \infty}  f_0(p,m,n')=1 =
\lim_{p\to \infty}  f_0(p,m,n')\,;$$
the first equality characterizes $C_p$ as the limit of the probability that a random 
$m \times m$ integer matrix over $\mathbb{Z}/p^s\mathbb{Z}$ is nonsingular as $m\to \infty$\,.
\end{theorem}
\begin{proof}
Utilizing the expression on the right-hand side of \eqref{eq:f0}, we obtain the monotonicities. 
Then we apply  Lemma \ref{lemma:Cp} to get
$$\inf_{m} f_0(p,m,n')
=\lim_{m\to \infty} f_0(p,m,n')
=\frac{C_p}{[p,n']}\to \frac{C_p}{C_p}=1 \quad {\rm{as}}\ \ n'\to\infty\,,$$
and
$$\inf_{m,n'} f_0(p,m,n')= \lim_{m\to \infty} f_0(p,m,0)
=C_p\to 1\quad {\rm{as}}\ \ p\to\infty\,.$$
\end{proof}

\subsubsection{\texorpdfstring{The Case of\, ${\bm b}\neq {\bm 0}\,$}{The Case of b!=0}}
$ $ 

We first present the monotonicity properties of $f(p,s,m,n',{\bm b})$ in $n'$ and $m$.

\begin{theorem}\label{thm:b!=0}
Suppose that ${\bm b}\neq {\bm 0}$\,. 
The function $f(p,s,m,n',{\bm b})$ is strictly decreasing in $n'$ while strictly increasing in $m$.
\end{theorem}
\begin{proof}
Recall that $b_1\ge b_2\ge \cdots b_s\ge 0$\,. Since ${\bm b}\neq {\bm 0}$, we have $b_1\ge 1$\,. Thus the ratio
\begin{align*}
&\ \frac{f(p,s,m,n'+1,{\bm b})}{f(p,s,m,n',{\bm b})}
=p^{- \sum_{i=1}^{s}(n'+1+b_i)b_i+ \sum_{i=1}^{s}(n'+b_i)b_i}\cdot 
\frac{[p,n'+1+m][p,n'+b_s]}
{[p,n'+m][p,n'+1+b_s]}\\
=&\ p^{- \sum_{i=1}^{s} b_i} \cdot \frac{ 1-p^{-(n'+1+m)} }
{1-p^{-(n'+1+b_s)}}
< p^{-1}\cdot  \frac{1}{1-p^{-1}}=\frac{1}{p-1}\le 1 \,,
\end{align*}
and
\begin{align*}
&\ \frac{f(p,s,m+1,n',{\bm b})}{f(p,s,m,n',{\bm b})}
=\frac{[p,n'+m+1][p,m+1][p,m-b_1]}
{[p,n'+m][p,m][p,m+1-b_1]}\\
=&\ \frac{ \left( 1-p^{-(n'+1+m)} \right) \left( 1-p^{-(m+1)} \right) }
{1-p^{-(m+1-b_1)}}
\ge  \frac{\left( 1-p^{-(m+1)} \right)^2}{1-p^{-m}}
>\frac{1-2p^{-(m+1)}}{1-p^{-m}}
\ge 1 \,,
\end{align*}
as $p\ge 2$\,.
\end{proof}

Recall from definition \eqref{eq:f} that $f$ is the product of a power of $p$
\begin{equation*} 
f_1(p,s,m,n',{\bm b}):=p^{- \sum_{i=1}^{s}(n'+b_i)b_i}
\end{equation*}
and a fraction
\begin{equation} \label{eq:f2}
f_2(p,s,m,n',{\bm b}):=\frac{[p,n'+m][p,m]}
{[p,n'+b_s][p,b_s] \prod_{i=1}^{s} [p,b_{i-1}-b_i]}\,.
\end{equation}

When $s$ is fixed, thanks to Lemma \ref{lemma:Cp}, the function $f_2$ defined in \eqref{eq:f2} is bounded regardless of the values of other variables. Moreover, when $m$ (instead of $s$)  is fixed, this result also holds  since $\sum_{i=1}^{s} (b_{i-1}-b_i)=m-b_s\le m$ implies that 
$$\prod_{i=1}^{s} [p,b_{i-1}-b_i]
=\prod_{i=1}^{s}\prod_{j=1}^{b_{i-1}-b_i}\left(1-p^{-j} \right) 
\ge \left(1-p^{-1}\right) ^m\ge 2^{-m}\,.$$ 
These observations lead to the following zero limiting probabilities.

\begin{theorem}\label{thm:lim0}
We have
\begin{equation*}
\lim_{\max\,\left\{p,\,n',\, \sum_{i=1}^{s}b_i\right\}\to \infty,\,{\tiny {\bm b}}\neq{\tiny {\bm 0}}} \max_{m} \,f(p,s,m,n',{\bm b})=0\quad \text{when}\ s\  \text{is\ fixed}
\end{equation*}
and
\begin{equation*}
\lim_{\max\,\left\{p,\,n',\, \sum_{i=1}^{s}b_i\right\}\to \infty,\,{\tiny {\bm b}}\neq{\tiny {\bm 0}}} \,f(p,s,m,n',{\bm b})=0\quad \text{when}\ m\  \text{is\ fixed.}
\end{equation*}
\end{theorem}
\begin{proof}
When $s$ or $m$ is fixed, we have shown that $f_2$ is bounded. On the other hand, we have
$$f_1(p,s,m,n',{\bm b})=p^{- \sum_{i=1}^{s}(n'+b_i)b_i}=p^{- n'\sum_{i=1}^{s}b_i-\sum_{i=1}^{s}b_i^2}\to 0$$ 
as long as ${\bm b}\neq {\bm 0}$ and 
\begin{equation}\label{eq:infty}
\max\left\{p,n'\sum_{i=1}^{s}b_i+\sum_{i=1}^{s}b_i^2\right\}\to \infty\,.
\end{equation} 
Noticing that $$n'\sum_{i=1}^{s}b_i\le n'\sum_{i=1}^{s}b_i+\sum_{i=1}^{s}b_i^2
\le n'\sum_{i=1}^{s}b_i+\left(\sum_{i=1}^{s}b_i\right)^2\,,$$
thus \eqref{eq:infty} is equivalent to $\max\left\{p,n',\sum_{i=1}^{s}b_i\right\}\to \infty$\,.
\end{proof}

\begin{remark}\label{rmk:r}
{\rm
Let $r\,(\le s)$ be the number of nonzeroes in $\{b_1,b_2,\dots,b_s\}$, i.e., $b_r>0=b_{r+1}$ 
(we define $b_{s+1}=0$), then $rb_r\,,b_1\le \sum_{i=1}^{s}b_i\le  rb_1$ due to the decreasing property of the $b_i$'s.
Hence $\sum_{i=1}^{s}b_i\to \infty$ if and only if $\max\,\{b_1,r\}\to \infty$ \,.
In particular, when $s$ is fixed, we have $\sum_{i=1}^{s}b_i\to \infty$ if and only if $b_1\to \infty$.}
\end{remark}

Moreover, if we free $s,m,n'$ but fix $p$ and let $\sum_{i=1}^{s}b_i\to \infty$\,, then $f$ also goes to $0$.

\begin{theorem}\label{thm:lim0'}
For a fixed prime $p$, we have
\begin{equation*}
\lim_{ \sum_{i=1}^{s}b_i \to \infty} \max_{m,n'} \,f(p,s,m,n',{\bm b})=0\,.
\end{equation*}
\end{theorem}

\begin{proof} 
Since $\sum_{i=1}^{s}b_i \to \infty$\,, we can assume that ${\bm b}\neq {\bm 0}$\,. Moreover, if $r\,(\le s)$ is the number of nonzeroes in $\{b_1,b_2,\dots,b_s\}$, then $\max\,\{b_1,r\}\to \infty$ (see Remark \ref{rmk:r}), which is equivalent to that $b_1\to \infty$ or $r\to \infty$ holds.

\smallskip
\noindent
{\it Case 1.} $b_1\to \infty$\,.

For any fixed $p,s,m$ and $n$, from Lemma \ref{lemma:b_i equal} we see that for ${\bm b}'=(b_1,b_s,b_s,\dots,b_s)$,
\begin{align*}
& f(p,s,m,n',{\bm b})\le f(p,s,m,n',{\bm b}')\\
=&\ p^{-(n'+b_1)b_1-(s-1)(n'+b_s)b_s}\cdot
\frac{[p,n'+m][p,m]}
{[p,n'+b_s][p,b_s][p,m-b_1][p,b_1-b_s]}
\le p^{-b_1}\cdot\frac{1}{(e^{-2})^4}
\end{align*}
on the strength of Lemma \ref{lemma:Cp}.
Hence
$$\max_{m,n'} \,f(p,s,m,n',{\bm b})\le p^{-b_1}\cdot\frac{1}{(e^{-2})^4} \to 0\,,\quad
\text{as}\ \ b_1\to \infty\,. $$

\noindent
{\it Case 2.} $r\to \infty$\,.

For any fixed $p,s,m$ and $n$, from Lemma \ref{lemma:b_i equal} we see that for ${\bm b}'=(b_r,b_r,\dots,b_r,b_s,b_s,\dots,b_s)$ (with $r$ $b_r$'s and $(s-r)$ $b_s$'s),
\begin{align*}
& f(p,s,m,n',{\bm b})\le f(p,s,m,n',{\bm b}')\\
=&\ p^{-r(n'+b_r)b_r-(s-r)(n'+b_s)b_s}\cdot
\frac{[p,n'+m][p,m]}
{[p,n'+b_s][p,b_s][p,m-b_r][p,b_r-b_s]}
\le p^{-r}\cdot\frac{1}{(e^{-2})^4}\,,
\end{align*}
on the strength of Lemma \ref{lemma:Cp}.
Hence
$$\max_{m,n'} \,f(p,s,m,n',{\bm b})\le p^{-r}\cdot\frac{1}{(e^{-2})^4} \to 0\,,\quad
\text{as}\ \ r\to \infty\,. $$
\end{proof}

All the limits of $f$ we have found so far equal  zero. 
To attain a nonzero limit, we must have a bounded
$\max\left\{p,n',\sum_{i=1}^{s}b_i\right\}$.
We may fix $p,s,n',{\bm b}$, let $m\to \infty$ and apply Lemma \ref{lemma:Cp}\,.
\begin{theorem}\label{thm:non0lim}
For fixed $p,s,n'$ and ${\bm b}\neq {\bm 0}$\,, we have 
\begin{equation*}
\lim_{m\to \infty} \,f(p,s,m,n',{\bm b})=p^{- \sum_{i=1}^{s}(n'+b_i)b_i}\cdot \frac{C_p}
{[p,n'+b_s][p,b_s] \prod_{i=2}^{s} [p,b_{i-1}-b_i]}\,.
\end{equation*}
\end{theorem}

We may also weaken the constraints by fixing $p,n'$ and $\sum_{i=1}^{s}b_i$ only.  A natural way to achieve this is to fix the first few $b_i$'s, say $b_1,b_2,\dots,b_r$ ($r< s$ fixed), and set the rest to be zero no matter how big $s$ is. 
According the definition \eqref{eq:f} of $f$, 
for ${\bm b}=(b_1,b_2,\dots,b_r,0,0,\dots,0)$,  we have
\begin{equation}\label{eq:r}
f(p,s,m,n',{\bm b})
=p^{-\sum_{i=1}^{r}(n'+b_i)b_i}\cdot
\frac{[p,n'+m][p,m]}
{[p,n'] \prod_{i=1}^{r+1} [p,b_{i-1}-b_i]}\,,
\end{equation}
which is independent of $s$.
Coupling with Theorem \ref{thm:lim0'} gives the following.

\begin{theorem}\label{thm:non0lim'}
When $m, n'$ and $p$ are fixed, for any given infinite integer sequence $\{b_0,b_1,\dots\}$ with $m=b_0\ge b_1\ge\cdots \ge b_i\ge b_{i+1}\ge\cdots \ge 0$\,, we have 
\begin{equation*}
\lim_{s\to \infty} \,f(p,s,m,n',{\bm b}^s)=
\begin{cases} 
0\,, & {\mathrm{if}}\ \sum_{i=1}^{\infty}b_i \to \infty \\ 
p^{-\sum_{i=1}^{r}(n'+b_i)b_i}\cdot
\frac{[p,n'+m][p,m]}
{[p,n'] \prod_{i=1}^{r+1} [p,b_{i-1}-b_i]}\,, & {\mathrm{otherwise}}
\end{cases},
\end{equation*}
where ${\bm b}^s:=(b_0,b_1,\dots,b_s)$ and  in the second case, $r$ is the number of nonzeroes in $\{b_0,b_1,\dots\}$ and finite (see Remark \ref{rmk:r}), and $b_{r+1}=0$\,.
\end{theorem}

\end{document}